\documentclass[a4paper,10pt,english,french]{smfart}

\usepackage[T1]{fontenc}        
\usepackage[latin1]{inputenc}   
\usepackage[francais, english]{babel}
\usepackage{lmodern}
\usepackage{amssymb,amsfonts,amsthm,amsmath,latexsym}
\usepackage{icomma}
\usepackage{enumerate}

\usepackage[margin=1.1in]{geometry}

\theoremstyle{plain}
\newtheorem{theoreme}{Théorème}[section]

\newtheorem{lemme}{Lemme}[section]
\newtheorem{prop}{Proposition}[section]
\newtheorem{cj}{Conjecture}[section]

\theoremstyle{definition}
\newtheorem*{ack}{Remerciements}

\newcommand{\e}{{\rm e}}
\newcommand{\dd}{{\rm d}}

\newcommand{\cC}{{\mathcal C}}
\newcommand{\cPhi}{{\check{\Phi}}}

\newcommand{\cD}{{\mathcal D}}
\newcommand{\ee}{{\varepsilon}}
\newcommand{\kk}{{\mathbf k}}
\newcommand{\bfN}{{\mathbf N}}

\renewcommand\Re{\operatorname{\mathfrak{Re}}}

\numberwithin{equation}{section}

\title{Sur les solutions friables de l'équation~$a+b=c$}
\author{Sary Drappeau}
\date{\today}

\address{Université Denis Diderot (Paris VII) \\ Institut de Mathématiques de Jussieu \\ UMR 7586 \\ Bâtiment Chevaleret \\ Bureau~$7$C$08$ \\ $75205$ Paris Cedex 13}
\email{drappeau@math.jussieu.fr}

\usepackage[unicode]{hyperref}

\begin{abstract}
Dans un récent article~\cite{SoundLaga2011}, Lagarias et Soundararajan étudient les solutions friables à l'équation a+b=c. Sous l'hypothèse de Riemann généralisées aux fonctions L de Dirichlet, ils obtiennent une estimation pour le nombre de solutions pondérées par un poids lisse et une minoration pour le nombre de solutions non pondérées. Le but de cet article est de présenter des arguments qui permettent d'une part de préciser les termes d'erreur et d'étendre les domaines de validité de ces estimations afin de faire le lien avec un travail de la Bretèche et Granville~\cite{AGRB2010}, d'autre part d'obtenir le comportement asymptotique exact du nombre de solutions non pondérées. 
\end{abstract}

\begin{altabstract}
In a recent paper~\cite{SoundLaga2011}, Lagarias and Soundararajan study the~$y$-smooth solutions to the equation~$a+b=c$. Under the Generalised Riemann Hypothesis, they obtain an estimate for the number of those solutions weighted by a 
compactly supported smooth function, as well as a lower bound for the number of bounded unweighted solutions. In this paper, we aim to prove a more precise estimate for the number of weighted solutions that is valid
when~$y$ is relatively large with respect to~$x$, so as to connect our estimate with the one obtained by La Bretèche and Granville in a recent work~\cite{AGRB2010}. We also prove the conjectured upper bound for the number
of bounded unweighted solutions, thus obtaining its exact asymptotic behaviour.
\end{altabstract}

\begin{document}

\maketitle

\tableofcontents

\section{Introduction}

La conjecture~$abc$ formulée en 1985 par Masser et Oesterlé (voir par exemple~\cite{Oest88}) relie la taille des solutions à l'équation
\begin{equation}\label{abc}
a + b = c
\end{equation}
avec leur \emph{radical}~$R(a, b, c) = \prod_{p | a b c}{p}$ de la façon suivante.
\begin{cj}
Il existe une constante~$\kappa$ telle que pour tout~$\ee > 0$, il n'existe
qu'un nombre fini de solutions~$(a, b, c)$ avec~${\rm pgcd}(|a|, |b|, |c|) = 1$ à l'équation
\eqref{abc} sous la condition
\[ R(|a|, |b|, |c|) \le \max(|a|, |b|, |c|)^{\kappa-\ee}. \]
\end{cj}
On peut approcher le problème en comptant les solutions de~\eqref{abc} suivant la taille de leurs facteurs premiers.
C'est l'objet d'un travail de Lagarias et Soundararajan~\cite{SoundLaga2011, SoundLagaXYZ}.
On note respectivement~$P^+(n)$ et~$P^-(n)$ le plus grand et le plus petit facteur premier de~$n$ (avec les conventions~$P^+(1) = 1$ et~$P^-(1) = \infty$), ainsi que
\[ S(x, y) = \{ n \in {\mathbf N} \mid 1 \le n \le x \mbox{ et } P^+(n) \le y \} \]
\[ \Psi(x, y) = \left| S(x, y) \right| \]
\[ u  = u_x := \frac{\log x}{\log y} \]
\[ H(u) := \exp(u / \log^2 (u + 1)) \text{\quad pour~$u \geq 1$.} \]
On note~$\alpha = \alpha_x = \alpha(x, y)$ (le \emph{point-selle}) l'unique solution positive de
\[ \sum_{p \le y}{\frac{\log(p)}{p^{\alpha}-1}} = \log x. \]
Là où~$\alpha$ et~$u$ sont notés sans indice, il est convenu que la valeur du premier paramètre est~$x$. La valeur du second paramètre,
elle, est toujours~$y$. Lorsque~$x$ et~$y$ tendent vers l'infini tout en vérifiant~$(\log x)^2 \leq y \leq x$, on a
\[ 1 - \alpha \sim (\log u)/\log y \].
On étudie le comportement asymptotique du nombre de solutions à l'équation~\eqref{abc} qui sont $y$-friables et de taille inférieure à~$x$
\[ N(x, y) := \sum_{\substack{a, b, c \in S(x, y) \\ a+b=c }}{1} \]
ainsi que du nombre de ces solutions qui sont \emph{primitives}, c'est-à-dire avec à composantes premières entre elles
\[ N^*(x, y) := \sum_{\substack{a, b, c \in S(x, y) \\ (a, b, c) = 1 \\ a+b=c }}{1} .\]
Le comportement asymptotique de~$N(x, y)$ a été étudié par de la Bretèche et Granville dans~\cite{AGRB2010} pour les grandes valeurs de~$y$.
\begin{theoreme}[\cite{AGRB2010}, théorème 1.1]\label{BGthm}
Soit~$\ee>0$. Lorsque~$x$ et~$y$ vérifient
\[ \exp((\log x)^{2/3 + \ee}) \leq y \leq x \]
on a uniformément
\[ N(x, y) = \frac{1}{2} \frac{\Psi(x, y)^3}{x} \left( 1 + O_\ee \left( \frac{\log (u+1)}{\log y} \right) \right) .\]
\end{theoreme}
Dans le cadre de l'étude de la conjecture~$abc$ cependant, il apparaît intéressant d'obtenir des résultats valables
lorsque~$y$ est de l'ordre d'une puissance de~$\log x$. On travaille pour cela avec des sommes modifiées, où la condition~$a, b, c \leq x$ est lissée.
\'{E}tant donnée une fonction test $\Phi : {\mathbf R}^+ \rightarrow {\mathbf C}$ à support compact, on étudie le comportement du nombre de solutions pondérées
\[ N(x, y ; \Phi) := \sum_{\substack{P^+(abc) \leq y \\
								    a+b=c }}{\Phi \left( \frac{a}{x} \right)
								             \Phi \left( \frac{b}{x} \right)
								             \Phi \left( \frac{c}{x} \right) } \]
ainsi que la quantité associée pour les solutions primitives
\[ N^*(x, y ; \Phi) := \sum_{\substack{P^+(abc) \leq y \\
									(a, b, c) = 1 \\
								    a+b=c }}{\Phi \left( \frac{a}{x} \right)
								             \Phi \left( \frac{b}{x} \right)
								             \Phi \left( \frac{c}{x} \right) }. \]
La fonction~$\Phi$ est autant que possible utilisée comme une approximation de~${\mathbf 1}_{]0,1]}$,
la fonction indicatrice de l'intervalle~$]0,1]$.

De même que dans~\cite{SoundLaga2011}, on se place dans la situation où l'hypothèse de Riemann est vraie dans la version généralisée suivante.
\begin{cj}{(Hypothèse de Riemann généralisée)}
Les zéros non triviaux de la fonction~$\zeta$ de Riemann ainsi que ceux de toutes les
fonctions~$L$ de Dirichlet ont pour partie réelle~$1/2$.
\end{cj}
Si~$2/3 < \alpha \leq 1$ et si~$\Phi$ est de classe~$\cC^\infty$ par morceaux, à support compact dans~${\mathbf R}_+$, alors on définit
\[ {\mathfrak S}_{0}(\Phi, \alpha) := \alpha^3 \int_{0}^{\infty} {
        \int_{0}^{\infty} {\Phi(t_1) \Phi(t_2) \Phi(t_1+t_2)
        (t_1 t_2 (t_1 + t_2))^{\alpha - 1}\dd t_1} \dd t_2} \]
\[ {\mathfrak S}_{1}(\alpha) := \prod_{p} { \left( 1 + \frac{p-1}{
        p (p^{3 \alpha - 1} - 1)} \left( \frac{p-p^\alpha}{p-1} \right)^3
        \right) } .\]
On note également~${\mathfrak S}^*_{1}(\alpha, y) := {\mathfrak S}_{1}(\alpha) \zeta(3\alpha-1, y)^{-1}$ et ${\mathfrak S}^*_{1}(\alpha) := {\mathfrak S}_{1}(\alpha) \zeta(3\alpha-1)^{-1}$,
avec la notation, pour tout~$s \in {\mathbf C}$ de partie réelle strictement positive,
\[ \zeta(s, y) := \sum_{P^+(n) \leq y}{n^{-s}} = \prod_{p \leq y}{\left(1-p^{-s}\right)^{-1}} .\]
Remarquons que~${\mathfrak S}_{0}(\Phi, \alpha)$ et~${\mathfrak S}_{1}(\alpha)$ sont des fonctions continues de~$\alpha$, on a donc en particulier
\[ {\mathfrak S}_{0}({\mathbf 1}_{]0,1]}, \alpha){\mathfrak S}_{1}(\alpha) = \frac{1}{2} + O\left(\frac{\log u}{\log y}\right) \]

\begin{theoreme}[\cite{SoundLaga2011}, théorèmes 2.1 et 2.2]\label{SLthm}
Soit~$\ee>0$ et~$\Phi$ une fonction de classe~$\cC^\infty$ à support compact inclus dans~$]0, \infty[$.
Si l'hypothèse de Riemann généralisée est vraie, alors pour tous~$x$ et~$y$ vérifiant
\[ 2 \leq (\log x)^{8+\ee} \leq y \leq \exp((\log x)^{\frac{1}{2}-\ee})\]
on a
\[ N(x, y ; \Phi) = {\mathfrak S}_{0}(\Phi, \alpha) {\mathfrak S}_{1}(\alpha) \frac{\Psi(x, y)^3}{x}
                    \left\{ 1 + O_{\ee, \Phi}\left(\frac{\log \log y}{\log y} \right)  \right\} \]
\[ N^*(x, y ; \Phi) = {\mathfrak S}_{0}(\Phi, \alpha) {\mathfrak S}^*_{1}(\alpha, y) \frac{\Psi(x, y)^3}{x}
					\left\{ 1 + O_{\ee, \Phi}\left(\frac{1}{(\log y)^{1/4}} \right)  \right\} .\]
\end{theoreme}

Les termes principaux des estimations des Théorèmes~\ref{SLthm} et~\ref{BGthm} sont donc compatibles en vertu de la remarque précédente ;
mais leurs intervalles de validités ne se rejoignent pas.
En premier lieu, on présente ici une modification de l'argument de Lagarias et Soundararajan
qui permet d'une part de combler cet intervalle non résolu, d'autre part d'améliorer le terme d'erreur.
Définissons pour tout~$\kappa \geq 1$ et~$c_0>0$ les domaines
\begin{equation}\label{domaine_De}
\cD(\kappa) := \left\{ (x, y) \in {\mathbf R}^2 \mid 2 \le (\log x)^\kappa \le y \le x \right\}
\end{equation}
\begin{equation}\label{domaine_De*}
\cD^*(\kappa, c_0) := \cD(\kappa) \cap \left\{ (x, y) \in {\mathbf R}^2 \mid (\log y) H(u)^{-c_0} \leq 1 \right\}
\end{equation}
Notre résultat est le suivant.

\begin{theoreme}\label{th1}
Il existe une constante absolue~$c_0>0$ telle que pour tous~$\ee > 0$ et~$\Phi$ de classe~$\cC^\infty$ à support compact inclus dans~$]0, \infty[$,
si l'hypothèse de Riemann généralisée est vraie, alors pour tout~$(x, y)$ dans le domaine~$\cD^*(8+\ee, c_0)$ on a
\begin{equation}\label{eqv_pcp}
N(x, y ; \Phi) = {\mathfrak S}_{0}(\Phi, \alpha) {\mathfrak S}_{1}(\alpha)
                 \frac{\Psi(x, y)^3}{x} \left\{ 1 + O_{\ee, \Phi}\left(\frac{1}{u} \right) \right\}
\end{equation}
\begin{equation}\label{eqv_pcp_prim}
N^*(x, y ; \Phi) = {\mathfrak S}_{0}(\Phi, \alpha) {\mathfrak S}^*_{1}(\alpha, y)
                   \frac{\Psi(x, y)^3}{x} \left\{ 1 + O_{\ee, \Phi}\left(\frac{1}{u} \right) \right\} .
\end{equation}
\end{theoreme}

On s'attend, mais ce ne sera pas étudié ici, à ce que la technique de La Bretèche et Granville~\cite{AGRB2010} puisse être utilisée pour obtenir une estimation de~$N(x, y ; \Phi)$ valable sous les hypothèses du théorème~\ref{BGthm}.
Cela impliquerait que les estimations~\eqref{eqv_pcp} et~\eqref{eqv_pcp_prim} soient valables dans~$\cD(8+\ee)$ tout entier.
On note que, sous les hypothèses du Théorème~\ref{th1}, lorsque~$x$ et~$y$ tendent vers l'infini,
\[ 1 - \frac{1}{8 + \ee} + o(1) \leq \alpha < 1 .\]
Les termes d'erreurs du Théorème~\ref{th1} sont de même nature que le terme d'erreur identique obtenu par Hildebrand et Tenenbaum~\cite{TeneHild86}
dans l'estimation de~$\Psi(x, y)$ par la méthode du col. 

L'hypothèse sur~$\Phi$ est contraignante en cela qu'elle ne permet pas de prendre des majorants
de la fonction~${\mathbf 1}_{]0,1]}$ par le haut. On a en revanche comme corollaire direct les minorations asymptotiques suivantes.
\begin{theoreme}[\cite{SoundLaga2011}, corollaire des théorèmes 2.1 et 2.2]
Soit~$\ee>0$. Si l'hypothèse de Riemann est vraie, lorsque~$x$ et~$y$ tendent vers l'infini en vérifiant
\[ (\log x)^{8+\ee} \leq y \leq \exp((\log x)^{\frac{1}{2}-\ee}) \mbox{,\quad on a}\]
\[ N(x, y) \geq {\mathfrak S}_{0}({\mathbf 1}_{]0,1]}, \alpha) {\mathfrak S}_{1}(\alpha) \frac{\Psi(x, y)^3}{x}
                \left\{1 + o_\ee(1) \right\} \]
\[ N^*(x, y) \geq {\mathfrak S}_{0}({\mathbf 1}_{]0,1]}, \alpha) {\mathfrak S}^*_{1}(\alpha) \frac{\Psi(x, y)^3}{x}
                  \left\{1 + o_\ee(1) \right\}. \]
\end{theoreme}
On présente un raisonnement qui permet de parvenir aux égalités asymptotiques.
\begin{theoreme}\label{th2}
Soit~$\ee>0$. Si l'hypothèse de Riemann est vraie, lorsque~$(x, y) \in \cD(8+\ee)$ et~$x$ et~$y$ tendent vers l'infini on a
\begin{equation}\label{eqv_nonpond} N(x, y) = {\mathfrak S}_{0}({\mathbf 1}_{]0,1]}, \alpha) {\mathfrak S}_{1}(\alpha) \frac{\Psi(x, y)^3}{x} \left\{1 + o_\ee(1) \right\} \end{equation}
\begin{equation}\label{eqv_nonpond_prim} N^*(x, y) = {\mathfrak S}_{0}({\mathbf 1}_{]0,1]}, \alpha) {\mathfrak S}^*_{1}(\alpha) \frac{\Psi(x, y)^3}{x} \left\{1 + o_\ee(1) \right\} . \end{equation}
\end{theoreme}

Deux éléments ont permis d'améliorer les résultats de Lagarias et Soundararajan : un changement dans le choix d'un contour,
et l'utilisation de résultats très précis de La Bretèche et Tenenbaum~\cite{TeneRDLB2005Stat} sur le rapport $\Psi(x/d, y)/\Psi(x,y)$.

\begin{ack}
L'auteur souhaite adresser ses meilleurs remerciements à son directeur de thèse, Régis de la Bretèche, pour sa grande disponibilité et ses nombreux conseils durant la rédaction de cet article,
ainsi qu'à Andrew Granville pour sa relecture et ses remarques.
\end{ack}

\section{Rappel de quelques résultats sur les entiers friables}

Dans leur article~\cite{TeneHild86}, Hildebrand et Tenenbaum ont utilisé la méthode du col afin d'étudier~$\Psi(x, y)$, en tirant avantage de l'identité
\begin{equation}\label{expr_psi_integ}
\Psi(x, y) = \frac{1}{2 i \pi} \int_{\sigma-i\infty}^{\sigma+i\infty}{ \zeta(s, y) x^s \frac{\dd s}{s}}
\end{equation}
valable pour tout~$\sigma > 0$. Il apparaît en effet que si l'on choisit~$\sigma = \alpha(x, y)$ alors la contribution principale à l'intégrale entière vient
de la partie du contour autour de~$\alpha$, c'est-à-dire les valeurs de~$s$ ayant une petite partie imaginaire.

On note~$s \mapsto \phi_2(s, y)$ la dérivée seconde de la fonction~$s \mapsto \log\zeta(s, y)$.
Le résultat principal de Hildebrand et Tenenbaum est contenu dans les deux lemmes suivants.

\begin{lemme}[\cite{TeneHild86}, lemme 10]\label{ht_psi_integ}
Soit~$\ee > 0$. Il existe~$c_1>0$ tel que lorsque~$(x, y) \in \cD(1)$ on ait
\[
\Psi(x, y) = \frac{1}{2 i \pi} \int_{\alpha-i/\log y}^{\alpha+i/\log y}{\zeta(s ; y)\frac{x^s}{s} \dd s}
+ O_{\ee}\left(x^\alpha \zeta(\alpha ; y) \left( H(u)^{-c_1} +
                                            \exp\left\{-(\log y)^{3/2-\ee}\right\} \right) \right)
.\]
\end{lemme}
\begin{lemme}[\cite{TeneHild86}, lemme 11]\label{ht_integ_asympt}
Lorsque~$(x, y) \in \cD(1)$ on a
\begin{align*}
\frac{1}{2 i \pi} \int_{\alpha-i/\log y}^{\alpha+i/\log y}{\zeta(s ; y)\frac{x^s}{s} \dd s} & =
\frac{1}{2 i \pi} \int_{\alpha-i/\log y}^{\alpha+i/\log y}{ \left| \zeta(s ; y)\frac{x^s}{s} \right| |\dd s|}
   \left( 1 + O\left( \frac{1}{u} \right) \right) \\
& = \frac{x^\alpha \zeta(\alpha ; y)}{\alpha \sqrt{2 \pi \phi_2(\alpha, y)}} \left(1 + O\left(\frac{1}{u}\right)\right)
.\end{align*}
\end{lemme}
On dispose par ailleurs du résultat élémentaire suivant.
\begin{lemme}\label{lemme_zeta_u}
Lorsque~$x \geq y \geq 2$, on a
\[ \log \zeta(\alpha, y) = u\left\{1 + O\left(\frac{\log \log (u+2)}{\log (u+2)}\right)\right\} .\]
\end{lemme}
Cela permet d'écrire dans le même domaine
\begin{equation}\label{psi_x_alpha}
\Psi(x, y) = x^{\alpha + o(1)}
.\end{equation}

On reprend dans la section suivante une partie de la preuve de Hildebrand et Tenenbaum, on cite donc quelques résultats précis concernant l'estimation de l'intégrale~\eqref{expr_psi_integ}.
On se place sur la droite~$\Re s = \alpha$ et on note~$T_0 = u^{-1/3}/\log y$. Les points de partie imaginaire supérieure à~$T_0$ contribuent de façon négligeable.
Quant aux autres, un développement limité permet d'estimer leur contribution.
Le lemme suivant permet de traiter le cas des points de partie imaginaire~$\tau$ vérifiant $1/\log y \leq |\tau| \leq y$.

\begin{lemme}[\cite{TeneHild86}, lemme 8]\label{ht_majo_zeta}
Lorsque~$1/\log y \leq |\tau| \leq y$, il existe~$c_2>0$ tel que pour tout~$(x, y) \in \cD(1)$ on ait
\[ \zeta(\alpha + i \tau ; y) \ll \zeta(\alpha ; y) \exp \left( -c_2 \frac{u \tau^2}{(1 - \alpha)^2 + \tau^2} \right) .\]
\end{lemme}

Pour les points de partie imaginaire vérifiant $T_0 \leq |\tau| \leq 1/\log y$ on utilise la majoration suivante,
énoncée dans la démonstration du lemme~11 de~\cite{TeneHild86}.
\begin{lemme}[\cite{TeneHild86}, démonstration du lemme~11]\label{ht_majo_tau_mid}
Pour tout~$(x, y) \in \cD(1)$ et~$\sigma = \alpha$, on a
\[ \int_{T_0\leq|\tau|\leq1/\log y}{\left|\zeta(s, y)\frac{x^s}{s}\right|\dd s} \ll \frac{\Psi(x, y)}{u} .\]
\end{lemme}
Enfin, le lemme suivant est une reformulation du lemme~4 de~\cite{TeneHild86}.
\begin{lemme}[\cite{TeneHild86}, lemme~4]\label{ht_zeta_taylor}
On note~$\sigma_k$ ($2\leq k\leq 4$) la valeur de la dérivée~$k$-ième de la fonction~$s \mapsto \log \zeta(s, y)$ en~$s = \alpha$.
Pour~$(x, y) \in \cD(1)$ et~$s = \alpha + i\tau$ avec~$|\tau| \leq T_0$ on a
\[ \log\zeta(s, y) = \log\zeta(\alpha, y) - i\tau\log x - \frac{\tau^2}{2}\sigma_2 -i\frac{\tau^3}{6}\sigma_3 + O(\tau^4 \sigma_4) .\]
De plus les quantités~$\sigma_k$ vérifient $\sigma_k \asymp u (\log y)^k$.
\end{lemme}
On note que~$\sigma_2 = \phi_2(\alpha,y)$.

\bigskip

Concernant le comportement local de~$\Psi(x, y)$, on cite le résultat suivant, dû à La Bretèche et Tenenbaum~\cite{TeneRDLB2005Stat}.

\begin{lemme}[\cite{TeneRDLB2005Stat}, Théorème 2.4]\label{estim_psi_local}
Il existe deux constantes absolues~$b_1$ et~$b_2$  et une fonction~$b = b(x, y; d)$ satisfaisant~$b_1 \leq b \leq b_2$ telles que pour tout~$(x, y) \in \cD(1)$
et~$d \in {\mathbf N}$ avec~$1 \leq d \leq x$ on ait uniformément
\[
\Psi\left( \frac{x}{d}, y\right) = \left\{1 + O\left(\frac{t}{u}\right) \right\} \left(1 - \frac{t^2}{u^2}\right)^{b u} \frac{\Psi(x, y)}{d^\alpha}
\]
où on a posé $t = (\log d)/\log y$.
\end{lemme}
Remarquons que cela implique sous les mêmes hypothèses la majoration
\[ \Psi(x/d, y) \ll  \Psi(x, y)/d^\alpha .\]
On dispose avec ceci de tous les outils nécessaires pour préciser le résultat de Lagarias et Soundararajan.

\section{Solutions générales pondérées}

Dans ce qui suit, $\Phi : {\mathbf R_+} \rightarrow {\mathbf C}$ désigne une fonction de classe~$\mathcal{C}^{\infty}$ à support compact inclus dans~$]0, \infty [$.
L'objet d'étude dans cette partie est la quantité suivante.
\[ N(A, B, C, y ; \Phi) := \sum_{\substack{P^+(abc) \leq y \\
								    a+b=c }}{\Phi \left( \frac{a}{A} \right)
								             \Phi \left( \frac{b}{B} \right)
								             \Phi \left( \frac{c}{C} \right) } .\]
On a donc en particulier $N(x, y ; \Phi) = N(x, x, x, y ; \Phi)$.

\subsection{Rappel des lemmes de Lagarias et Soundararajan~\cite{SoundLaga2011}}

Dans un premier temps, on rappelle les principales étapes du raisonnement de Lagarias et Soundararajan.
Leur méthode se base sur la méthode du cercle.
On note $\e(x) = \exp(2 i \pi x)$ et on pose
\[ E_\Phi(x, y ; \vartheta) := \sum_{P^+(n) \leq y}{\e(n\vartheta) \Phi\left(\frac{n}{x}\right)} .\]
On peut alors écrire
\begin{equation}\label{mc_expr_integrale}
N(A, B, C, y ; \Phi) = \int_0^1{E_\Phi(A, y ; \vartheta) E_\Phi(B, y ; \vartheta) E_\Phi(C, y ; -\vartheta) \dd\vartheta}
.\end{equation}
Le but est de trouver des estimations de~$E_\Phi(x, y ; \vartheta)$ qui se comportent bien lorsqu'on les reporte dans l'intégrale.

Le comportement de~$E_\Phi(x, y ; \vartheta)$ est fortement lié aux approximations rationnelles de~$\vartheta$ : ainsi, si~$\vartheta = a/q + \beta$
on peut écrire (\emph{cf.} la démonstration de la proposition~6.1 de~\cite{SoundLaga2011})
\begin{equation}\label{expr_E_somme}
E_\Phi(x, y ; \vartheta) = \sum_{\substack{d | q \\ P^+(d) \leq y}}{\frac{1}{\varphi(q/d)}
                   \sum_{\chi \text{ (mod } q/d \text{)}}{
                      \chi(a) \tau(\overline{\chi}) \left( \sum_{P^+(m) \leq y}{\e (m d \beta) \chi(m) \Phi \left( \frac{m d}{x} \right)} \right)}}
\end{equation}
où~$\tau(\overline{\chi})$ est la somme de Gauss
$\sum_{b \text{ (mod } q/d \text{)}}{\overline{\chi}(b) \e(\frac{b}{q/d})}$.
La proposition 6.1 de~\cite{SoundLaga2011} montre que la contribution des caractères non principaux à~$E_\Phi(x, y ; \vartheta)$ est négligeable.
La version qu'on énonce ici se place sous des hypothèses plus générales, mais se montre de manière identique :
c'est pourquoi nous n'en donnons pas la preuve.
\begin{prop}[\cite{SoundLaga2011}, proposition~6.1]\label{SoundLaga2011_prop_6_1}
Soit~$\ee>0$ et~$\Phi$ une fonction de classe~$\mathcal{C}^{\infty}$ à support compact inclus dans~$]0, \infty [$.
Si l'hypothèse de Riemann généralisée est vraie, alors lorsque les réels~$x$, $y$, $R$ et~$\vartheta$ vérifient $2 \leq y \leq x \leq R$
et~$\vartheta \in [0, 1]$ avec~$\vartheta = a/q + \beta$ où $(a, q) = 1$, $q \leq R^{1/2}$  et~$|\beta| \leq 1/(q R^{1/2})$, on~a
\[ E_\Phi(x, y ; \vartheta) = M_\Phi(x, y ; q, \beta) + O_{\ee, \Phi}(x R^{-1/4 + \ee}) \]
où~$M_\Phi(x, y ; q, \beta)$ est défini par
\[ M_\Phi(x, y ; q, \beta) := \sum_{P^+(n) \leq y}{\frac{\mu(q / (q, n))}{\varphi(q / (q, n))} \e(n \beta) \Phi\left(\frac{n}{x}\right)} .\]
\end{prop}

Cette proposition ne donne pas une bonne majoration pour des~$y$ trop petits :
par exemple lorsque~$y < (\log x)^{4-\ee}$, la majoration triviale
\[ E_\Phi(x, y ; \vartheta) \ll_\Phi \Psi(x, y) = x^{\alpha + o(1)} \ll x^{3/4 - \ee/2} \]
est plus forte.

Le point important dans la démonstration de cette proposition est d'avoir une majoration efficace de la somme d'exponentielles
\[ \Psi_0(z, y ; \chi, \gamma) := \sum_{P^+(n) \leq y}{\e (n \gamma) \chi(n) \Phi \left( \frac{n}{z} \right)} \]
lorsque~$\chi$ est un caractère de Dirichlet et~$\gamma \in {\mathbf R}$.
On a pour tout~$\sigma>0$,
\[ \Psi_0(z, y ; \chi, \gamma) = \frac{1}{2 i \pi} \int_{\sigma-i\infty}^{\sigma+i\infty}{L(s, \chi ; y) z^s \cPhi(\gamma z, s) \dd s} \]
où l'on a noté
\[ L(s, \chi ; y) := \sum_{P^+(n) \leq y}{\chi(n) n^{-s}} = \prod_{p\leq y}{(1- \chi(p) p^{-s})^{-1}} \]
\[ \cPhi(\lambda, s) := \int_0^\infty{\Phi(t)\e(\lambda t)t^{s-1}\dd t} .\]
Afin de majorer~$\Psi_0$, il nous suffit d'avoir de bonnes majorations de~$L(s, \chi ; y)$ et~$\cPhi(\lambda,s)$ puis de les reporter dans l'intégrale.
Les majorations dont on dispose sont les suivantes.

\begin{prop}[\cite{SoundLaga2011}, proposition~5.1]\label{SoundLaga2011_prop_5_1}
Soit~$\ee>0$. Si l'hypothèse de Riemann généralisée est vraie, alors pour tout caractère~$\chi$ modulo~$q$
et tout~$s \in {\mathbf C}$ avec $s=\sigma+i\tau$ et~$1/2+\ee \leq \sigma \leq 3/2$, là où l'une quelconque des deux conditions suivantes est vérifiée :
\begin{itemize}
\item $\chi$ est non principal,
\item $\chi$ est principal et~$\tau > y^{1-\sigma}$,
\end{itemize}
on a
\[L(s, \chi ; y) \ll_{\ee} (q |\tau|)^\ee.\]
\end{prop}

\begin{prop}[\cite{SoundLaga2011}, lemmes~3.3 et~3.5]\label{SoundLaga2011_lemmas_33_35}
Soit~$k\geq 0$ et~$\Phi$ une fonction de classe~$\mathcal{C}^{\infty}$ à support compact inclus dans~$]0, \infty [$.
Pour tous~$\lambda \in {\mathbf R}$ et~$s \in {\mathbf C}$ que l'on écrit~$\sigma+i\tau$ avec~$\sigma \geq 1/4$ on a
\[ |\cPhi(\lambda, s)| \ll_{k, \Phi} \min\left( \left( \frac{1+|\lambda|}{|s|} \right)^k, \left( \frac{1+|s|}{|\lambda|} \right)^k \right) .\]
De plus, soit~$\ee>0$ et~$\delta \geq 0$. Lorsque~$\lambda \in {\mathbf R}$, on a
\[ \int_{-\infty}^{\infty}{|\cPhi(\lambda, s)|(1+|\tau|)^\delta \dd \tau} \ll_{\Phi, \sigma, \ee} (1+|\lambda|)^{1/2+\delta+\ee} .\]
\end{prop}

Il reste alors à traiter le terme principal~$M_\Phi(x, y ; q, \beta)$. On se ramene au cas~$P^+(q) \leq y$ en écrivant
$q = q_0 q_1$ avec~$P^+(q_0) \leq y$ et~$P^-(q_1) > y$. Il vient
\[ M_\Phi(x, y ; q, \beta) = \frac{\mu(q_1)}{\varphi(q_1)} M_\Phi(x, y ; q_0, \beta) .\]
Définissons
\[ \widetilde{M}_\Phi(x, y ; q_0, \beta) := \alpha q_0^{-\alpha} \prod_{p | q_0}{ \left( 1-\frac{p^\alpha-1}{p-1} \right) } \cPhi(\beta x, \alpha) \Psi(x, y) .\]
L'estimation de~$M_\Phi(x, y ; q_0, \beta)$ est l'objet de la proposition suivante de Lagarias et Soundararajan.
De même que la Proposition~\ref{SoundLaga2011_prop_6_1}, nous l'énonçons sous des hypothèses plus générales, mais la démonstration reste identique.

\begin{prop}[\cite{SoundLaga2011}, proposition~6.3]\label{SoundLaga2011_prop_6_3}
Soient~$\delta > 0$, $\ee>0$ et~$\Phi$ une fonction de classe~$\mathcal{C}^{\infty}$ à support compact inclus dans~$]0, \infty [$.
Si l'hypothèse de Riemann généralisée est vraie, alors lorsque les réels~$x, y, R, q_0$ et~$\beta$ vérifient
\[ 2 \leq (\log x)^{2+\ee} \leq y \leq x \leq R \mbox{, } q_0 \in S(R^{1/2}, y) \mbox{ et } \beta \in [-1/(q_0 R^{1/2}), 1/(q_0 R^{1/2})] \]
on a
\begin{itemize}

\item si~$x=R$ et~$|\beta| \geq R^{\delta-1}$ alors
\[ |M_\Phi(x, y ; q_0, \beta)| \ll_{\ee, \delta, \Phi} x^{3/4+\ee} q_0^{-1}, \]

\item si~$|\beta| \leq R^{\delta-1}$ et si de plus $y \leq \exp((\log x)^{1/2-\ee})$,
\[
M_\Phi(x, y ; q_0, \beta) = \widetilde{M}_\Phi(x, y ; q_0, \beta) + O_{\ee, \delta, \Phi} (x^{1/2+\ee} q_0^{-1} R^{1/4+\ee}) +  O_{\ee, \Phi}\left(\frac{q_0^{-\alpha+\ee} \Psi(x,y) }{(1+|\beta|x)^2 (\log y)} \right)
.\]

\end{itemize}
\end{prop}

Le reste du raisonnement consiste alors à reporter ces estimations dans l'intégrale~\eqref{mc_expr_integrale}.

On s'intéresse au deuxième cas de la Proposition~\ref{SoundLaga2011_prop_6_3},
car c'est lui qui dicte le terme d'erreur et la limite supérieure en~$y$ : $(\log y) H(u)^{-c_0} \leq 1$ du domaine de validité du Théorème~\ref{th1}.

\subsection{Estimation de \texorpdfstring{$M_\Phi(x, y ; q_0, \beta)$}{M}}

On démontre ici la proposition suivante, qui est une version plus précise de la proposition~6.3 de~\cite{SoundLaga2011}.
\begin{prop}\label{estim_m}
Il existe une constante~$c_0>0$ telle que si l'hypothèse de Riemann généralisée est vraie, pour tout~$\delta>0$ et toute fonction~$\Phi$ de classe~$\mathcal{C}^{\infty}$ à support compact inclus dans~$]0, \infty [$,
lorsque les réels~$x, y, R, q_0$ et~$\beta$ vérifient
\[ (x, y) \in \cD^*(1, c_0) \mbox{, } x \leq R \mbox{, } q_0 \in S(R^{1/2}, y) \mbox{ et } \beta \in [-R^{\delta-1}, R^{\delta-1}]\]
on ait
\begin{equation}\label{estim_m_psi}
M_\Phi(x, y ; q_0, \beta) = \widetilde{M}_\Phi(x, y ; q_0, \beta) + O_{\delta, \Phi} \left( x q_0^{-1/2} R^{-1/2+\delta} \right) + O_{\delta, \Phi} \left(  \frac{q_0^{-\alpha+\delta} \Psi(x,y)}{(1+|\beta|x) u} \right)
.\end{equation}
 
\end{prop}

\begin{proof}
Soit $c_0$ un réel vérifiant $0 < c_0 < \min(c_1/2, c_2/8)$, où $c_1$ et $c_2$ sont les constantes des lemmes~\ref{ht_psi_integ} et~\ref{ht_majo_zeta}. On choisit $x$, $y$, $R$, $q_0$ et $\beta$ comme dans l'énoncé.
On part de l'identité suivante valable pour tout~$\sigma>0$, obtenue en appliquant une transformation de Mellin,
\begin{equation}\label{expr_m_integ}
M_\Phi(x, y ; q_0, \beta) = \frac{1}{2i\pi}\int_{\sigma-i\infty}^{\sigma+i\infty}{\zeta(s ; y, q_0) x^s \cPhi(\beta x, s)\dd s }
\end{equation}
avec la notation
\[ \zeta(s ; y, q_0) = \sum_{P^+(n) \leq y}{\frac{\mu(q_0 / (q_0, n))}{\varphi(q_0 / (q_0, n)} n^{-s}} .\]
On remarque que l'on a
\[ \zeta(s ; y, q_0) = s q_0^{-s} \prod_{p | q_0}{\left(1-\frac{p^s-1}{p-1}\right)} \zeta(s ; y) .\]
On suit la méthode du col afin d'estimer l'intégrale qui apparaît dans~\eqref{expr_m_integ}. Posons $s = \sigma + i \tau$.
L'abscisse d'intégration est déplacée en~$\sigma = \alpha$ pour les $\tau$ petits.
Pour les plus grandes valeurs de~$\tau$, on se déplace progressivement vers la droite $\sigma = 1/2$ où le facteur $x^s$ a moins d'importance, en tirant avantage de l'hypothèse de Riemann.

Soit~$\ee>0$. On intègre sur le contour $\cup_{i=1}^{7}{\cC_j}$ où
\begin{itemize}
\item $\cC_1$ est la demi-droite $]1/2 + \ee - i \infty, 1/2 + \ee - i y^{1/2-\ee}]$,
\item $\cC_2$ est le chemin~$\left\lbrace 1 - (\log(-\tau)) / \log y + i \tau, \tau \in [-y^{1/2-\ee}, -y^{1-\alpha}] \right\rbrace$,
\item $\cC_3$ est le segment~$[ \alpha  - i y^{1-\alpha} ,  \alpha  - i / \log y]$,
\item $\cC_4$ est le segment~$[ \alpha  - i / \log y,  \alpha  + i / \log y]$,
\item $\cC_5$ est le segment~$[ \alpha  + i / \log y,  \alpha  + i y^{1-\alpha}]$,
\item $\cC_6$ est le chemin~$\left\lbrace 1 - (\log \tau)/\log y + i \tau, \tau \in [ y^{1-\alpha},  y^{1/2-\ee}] \right\rbrace$,
\item $\cC_7$ est la demi-droite $]1/2 + \ee - i \infty, 1/2 + \ee - i y^{1/2-\ee}]$,
\end{itemize}
chacun de ces chemins étant parcouru par parties imaginaires croissantes. Les chemins $\cC_j$ et~$\cC_{8-j}$ ($1 \leq j \leq 3$) étant conjugués, on se contentera de traiter
les chemin~$\cC_4$ à $\cC_7$. On note également
\[ I_j := \frac{1}{2 i \pi}\int_{\cC_j}{\zeta(s ; y, q_0) x^s \cPhi(\beta x, s)\dd s } \]
en remarquant que~$I_j = \overline{I_{8-j}}$ ($1 \leq j \leq 3$).

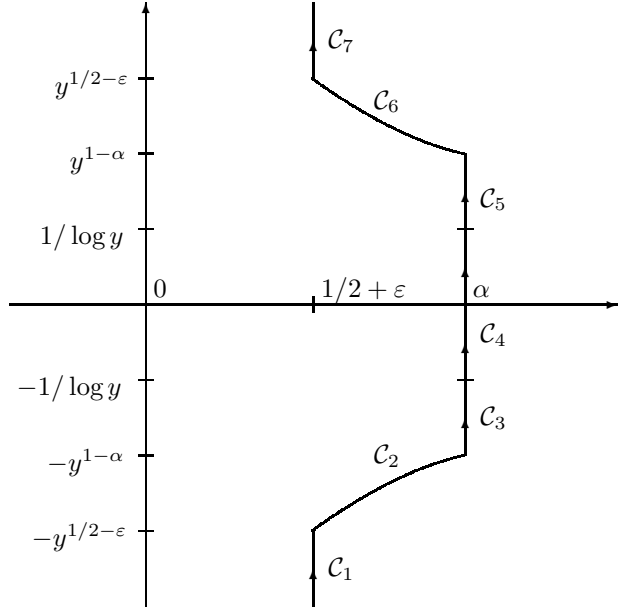
\begin{center}
\begin{picture}(80,80)
\linethickness{0.2mm}
\put(0, 40){ \vector(1, 0){80} }
\put(18, 0){ \vector(0, 1){80} }
\put(40, 39){ \line(0, 1){2} }
\put(17, 10){ \line(1, 0){2} }
\put(17, 20){ \line(1, 0){2} }
\put(17, 30){ \line(1, 0){2} }
\put(17, 50){ \line(1, 0){2} }
\put(17, 60){ \line(1, 0){2} }
\put(17, 70){ \line(1, 0){2} }

\put(59, 30){ \line(1, 0){2} }
\put(59, 50){ \line(1, 0){2} }

\put(40,  0){ \line(0, 1){10} }
\qbezier(41, 10)(52, 18)(61, 20)
\put(60, 20){ \line(0, 1){40} }
\qbezier(61, 60)(52, 62)(41, 70)
\put(40, 70){ \line(0, 1){10} }

\put(40,  0){ \vector(0, 1){ 5} }
\put(60, 20){ \vector(0, 1){ 5} }
\put(60, 30){ \vector(0, 1){ 5} }
\put(60, 40){ \vector(0, 1){ 5} }
\put(60, 50){ \vector(0, 1){ 5} }
\put(40, 70){ \vector(0, 1){ 5} }

\put( 6, 68){ $y^{1/2-\ee}$ }
\put( 8, 58){ $y^{1-\alpha}$ }
\put( 4, 48){ $1/\log y$ }
\put( 1, 28){ $-1/\log y$ }
\put( 5, 18){ $-y^{1-\alpha}$ }
\put( 3,  8){ $-y^{1/2-\ee}$ }
\put(19, 41){ $0$ }
\put(41, 41){ $1/2+\ee$ }
\put(61, 41){ $\alpha$ }

\put(42,  4){ $\cC_1$ }
\put(48, 19){ $\cC_2$ }
\put(62, 24){ $\cC_3$ }
\put(62, 35){ $\cC_4$ }
\put(62, 53){ $\cC_5$ }
\put(48, 66){ $\cC_6$ }
\put(42, 74){ $\cC_7$ }

\end{picture}
\end{center}

\bigskip

Sur le segment~$\cC_7$, les Lemmes~\ref{SoundLaga2011_prop_5_1} et~\ref{SoundLaga2011_lemmas_33_35} fournissent,
$$\zeta(s ; y) \ll_{\ee} q_0^{-1/2 + \ee} \tau^\ee \text{ et } x^s \ll x^{1/2+\ee} .$$
De plus,
\[ \int_{|\tau| > y^{1/2-\ee}}{ |\cPhi(\beta x, 1/2+\ee+\tau)| \tau^\ee } \ll_{\ee, \Phi} (1 + |\beta| x)^{1/2 + 2\ee}.\]
Donc quitte à prendre $\ee < \delta/(2 + 4\delta)$,
\[ I_1 \ll_{\delta, \Phi} x q_0^{-1/2} R^{-1/2+\delta}. \]

\medskip

Sur le segment~$\cC_6$, on intègre suivant $\sigma$.
On a
\[ I_6 = \frac{1}{2 i \pi} \int_{1/2 + \ee}^{\alpha}{
   \zeta(\sigma + i y^{1-\sigma} ; y, q_0) x^{\sigma + i y^{1-\sigma}} \cPhi(\beta x, \sigma + i y^{1-\sigma})
     (i (\log y) y^{1-\sigma} - 1) \dd \sigma} .\]
On a les majorations
\[ \zeta(s ; y, q_0) \ll_\delta q_0^{-\sigma + \delta/2} |\zeta(s ; y)|
   \ll_\delta q_0^{-\sigma+\delta} y^{\delta(1-\sigma)} \]
\[ \cPhi(\beta x, s) \ll_\Phi \frac{y^{1-\sigma}}{1 + |\beta| x} \]
\[ i (\log y) y^{1-\sigma} - 1 \ll (\log y) y^{1-\sigma} .\]
En reportant dans l'intégrale, on obtient pour une certaine constante $c_3 > 0$,
\begin{align*}
I_6 &\ \ll_{\delta, \Phi} \frac{q_0^{\delta} \log y}{1 + |\beta| x}
          \int_{1/2+\ee}^{\alpha}{ q_0^{-\sigma} x^\sigma y^{(2+\delta)(1-\sigma)} \dd \sigma} \\
    & \ll \frac{x^\alpha}{\log x} \frac{q_0^{-\alpha + \delta}}{1+|\beta|x}
          (\log y) y^{(2+\delta)(1-\alpha)} \\
    & \ll \frac{x^\alpha \zeta(\alpha ; y)}{\sqrt{u} \log y} \frac{q_0^{-\alpha + \delta}}{1+|\beta|x}
          \frac{\log y}{\sqrt{u}} (u \log u)^{2+\delta} \exp(-c_3 u) \\
    & \ll \frac{x^\alpha \zeta(\alpha ; y)}{\sqrt{u} \log y} \frac{q_0^{-\alpha + \delta}}{1+|\beta|x}
          \exp \left( -\frac{c_3}{2} u \right)
.\end{align*}
La contribution du chemin~$\cC_6$ est donc largement un terme d'erreur acceptable.

\medskip

La contribution du segment~$\cC_5$ s'écrit
\[ I_5 = \frac{1}{2 i \pi} \int_{1/\log y}^{y^{1-\alpha}}{
   \zeta( \alpha  + i \tau ; y, q_0) x^{ \alpha  + i \tau} \cPhi(\beta x, \alpha  + i \tau) \dd \tau} .\]
On utilise le Lemme~\ref{ht_majo_zeta}. On dispose pour~$\tau \in [1/\log y, y^{1-\alpha}]$ de la majoration
\begin{align*}
|\zeta(s ; y)| &\ \leq \zeta( \alpha  ; y) \exp \left( - c_2 u \frac{\tau^2}{(1- \alpha )^2 + \tau^2} \right) \\
& \leq \zeta( \alpha  ; y) \exp \left( - c_2 u \frac{1}{((1-\alpha)\log y)^2 + 1} \right) \\
& \leq \zeta( \alpha  ; y) H(u)^{-c_2/2}
.\end{align*}
On a donc
\begin{align*}
I_5 &\ \ll_\delta x^{ \alpha } \zeta( \alpha  ; y)
          H(u)^{-c_2/2} q_0^{- \alpha +\delta}
     \int_{1/(\log y)}^{y^{1-\alpha}}{|\cPhi(\beta x, \alpha  + i \tau)| \dd \tau} \\
    & \ll_{\delta, \Phi} x^{ \alpha } \zeta( \alpha  ; y)
          H(u)^{-c_2/2} q_0^{- \alpha +\delta}
     \frac{y^{2(1-\alpha)}}{1+|\beta| x}
.\end{align*}
L'intégrale sur $\cPhi$ est évaluée grâce au Lemme~\ref{SoundLaga2011_lemmas_33_35}, en distinguant suivant $|\beta| x \leq y^{1-\alpha}$ ou $|\beta| x > y^{1-\alpha}$.
Finalement, en utilisant~$y^{1-\alpha} \sim u \log u$ et d'après l'hypothèse~$c_0 < c_2/8$, on obtient
\begin{align*}
I_3 &\ \ll_{\delta, \Phi} \frac{x^{ \alpha } \zeta( \alpha  ; y)}{\sqrt{u} \log y} \frac{q_0^{- \alpha +\delta}}{1+|\beta| x}
          \exp \left( - c_2/2 \frac{u}{(\log u)^2} + 1/2\log u +  2 \log (u \log u) + \log \log y \right) \\
    & \ll \frac{x^{ \alpha } \zeta( \alpha  ; y)}{\sqrt{u} \log y} \frac{q_0^{- \alpha +\delta}}{1+|\beta| x}
           H(u)^{-c_2/8}
.\end{align*}
Ainsi~$I_5$ est de l'ordre du terme d'erreur annoncé.

\medskip

Sur le segment~$\cC_4$, la quantité à estimer est
\[ I_4 := \frac{1}{2 \pi} \int_{-1/\log y}^{1/\log y}{\zeta( \alpha  + i \tau ; y, q_0)
                                                    x^{ \alpha  + i \tau} \cPhi(\beta x,  \alpha  + i \tau) \dd \tau} .\]
On sépare l'intégrale en deux, selon la position de~$|\tau|$ par rapport à $T_0 = 1 / (u^{1/3} \log y)$.
D'après le Lemme~\ref{ht_majo_tau_mid}, on a
\begin{align*}
\frac{1}{2 i \pi} \int_{T_0 \leq |\tau| \leq 1/\log y}{ \zeta(s ; y, q_0) x^s \check{\Phi}(\beta x, s) \dd s}
& \ll_{\delta, \Phi} \frac{q_0^{-\alpha + \delta}}{1 + |\beta| x} \int_{T_0 \leq |\tau| \leq 1/\log y}{ |\zeta(s ; y) x^s | \dd s} \\
& \ll \Psi(x, y) \frac{q_0^{-\alpha + \delta}}{(1 + |\beta| x) u}
.\end{align*}
Pour l'autre partie de l'intégrale, le Lemme~\ref{ht_zeta_taylor} fournit le développement de Taylor suivant, valable pour~$|\tau| \leq T_0$,
\[ \zeta(s ; y) \frac{x^s}{s} = \frac{x^\alpha \zeta(\alpha ; y)}{\alpha} e^{-\tau^2 \sigma_2 / 2}
   \left( 1 - i \frac{\tau}{\alpha} - i \frac{\tau^3}{3!}\sigma_3 + O(\tau^6 \sigma_3^2 + \tau^2 + \tau^4 \sigma_4) \right) .\]
Pour traiter le facteur restant dans l'intégrale, on note
\[ f(\tau) := s q_0^{-s} \prod_{p | q_0}{\left( 1 - \frac{p^s - 1}{p - 1} \right) } \cPhi(\beta x, s) .\]
En remarquant que pour tout~$k \in {\mathbf N}$, l'application
$t \mapsto (\log t)^k \Phi(t)$ est également $\cC^\infty$ de même support que~$\Phi$, on déduit
\[ \cPhi(\beta x, \alpha) \ll_\Phi \frac{1}{1+|\beta|x} \]
\[ \frac{\partial\cPhi}{\partial s}(\beta x, \alpha) \ll_\Phi \frac{1}{1+|\beta|x} \]
\[ \sup_{|\tau| \leq T_0}\left| \frac{\partial^2\cPhi}{\partial s^2}(\beta x, s) \right| \ll_\Phi \frac{1}{1+|\beta|x} \]
ce qui entraîne
\[ |f(0)| + |f'(0)| + \sup_{|\tau| \leq T_0}|f''(\tau)| \ll_{\delta, \Phi} q_0^{-\alpha+\delta}/(1+|\beta|x) .\]
On a donc pour~$|\tau| \leq T_0$,
\[ f(\tau) = f(0) + f'(0) \tau + O_{\delta, \Phi}\left( \frac{q_0^{-\alpha+\delta}}{1+|\beta|x} \tau^2 \right). \]
Ainsi lorsque l'on multiplie ce développement limité avec celui de~$\zeta(s ; y) x^s / s$, on obtient
\begin{multline*}
\zeta(s ; y, q_0) x^s \cPhi(\beta x, s) = 
   \frac{x^\alpha \zeta(\alpha ; y)}{\alpha} e^{-\tau^2 \sigma_2 / 2} \\
   \times \left( f(0) + \lambda \tau + \mu \tau^3 + 
          O_{\delta, \Phi}\left((\tau^2 + \sigma_3^2 \tau^6 + \sigma_4 \tau^4) \frac{q_0^{-\alpha+\delta}}{1+|\beta|x}\right)
   \right)
\end{multline*}
où les coefficients $\lambda$ et~$\mu$ dépendent au plus de~$x$, $y$, $q_0$ et~$\delta$.
En intégrant cette expression pour~$|\tau| \leq T_0$, on obtient
\begin{multline*}
\frac{1}{2 i \pi} \int_{|\tau| \leq T_0}{ \zeta(s ; y, q_0) x^s \check{\Phi}(\beta x, s) \dd s} = 
 \frac{x^\alpha \zeta(\alpha ; y)}{2 \pi \alpha} f(0) \int_{|\tau| \leq T_0}{ e^{-\tau^2 \sigma_2 / 2} \dd \tau } \\
+ \frac{x^\alpha \zeta(\alpha ; y)}{2 \pi \alpha} \frac{q_0^{-\alpha+\delta}}{1+|\beta|x} \int_{|\tau| \leq T_0}{ e^{-\tau^2 \sigma_2 / 2}
   O_{\delta, \Phi}\left( \tau^2 + \sigma_3^2 \tau^6  + \sigma_4 \tau^4 \right) \dd \tau }
\end{multline*}
Le premier terme du membre de droite vaut
\[ f(0) \frac{x^\alpha \zeta(\alpha ; y)}{\alpha \sqrt{2 \pi \sigma_2}} \left( 1 + O\left(\frac{1}{u}\right) \right) .\]
Quant au deuxième, la majoration $\int_{{\mathbf R}}{|\tau|^k \exp(-\tau^2 \sigma_2 / 2) \dd \tau} \ll_k \sigma_2^{-(k+1)/2}$ permet d'écrire qu'il est
\begin{align*}
& \ll_{\delta, \Phi} \frac{x^\alpha \zeta(\alpha ; y)}{\sqrt{\sigma_2} \alpha} 
  \left( \sigma_2^{-1} + \sigma_3^2 \sigma_2^{-3} + \sigma_4 \sigma_2^{-2} \right) \frac{q_0^{-\alpha+\delta}}{1+|\beta|x}
 \ll \frac{x^\alpha \zeta(\alpha ; y) q_0^{-\alpha+\delta}}{\alpha u \sqrt{\sigma_2} (1+|\beta|x)}
\end{align*}
en utilisant l'approximation $\sigma_k \asymp u (\log y)^k$ ($2\leq k\leq 4$) énoncée au Lemme~\ref{ht_zeta_taylor}.

On obtient donc
\[
I_4 = \alpha q_0^{-\alpha} \prod_{p | q_0}{\left( 1 - \frac{p^{\alpha} - 1}{p - 1} \right) } \cPhi(\beta x, \alpha)
    \frac{x^\alpha \zeta(\alpha ; y)}{\alpha \sqrt{2 \pi \sigma_2}}
    + O_{\delta, \Phi} \left( \frac{q_0^{-\alpha + \delta} x^\alpha \zeta(\alpha ; y)}{u^{3/2} (\log y) (1 + |\beta| x)} \right)
.\]
On utilise maintenant les Lemmes~\ref{ht_psi_integ} et~\ref{ht_integ_asympt}, en vérifiant que les termes d'erreurs sont acceptables. 
Quitte à se restreindre à $\ee < 3/2$ et d'après l'hypothèse $c_0 < c_1/2$, où $c_1$ est la constante du Lemme~\ref{ht_psi_integ}, on a
\[ \exp((\log y)^{3/2 - \ee}) \gg u \sqrt{u} \log y \]
\[ H(u)^{c_1} \gg u \sqrt{u} H(u)^{{c_1}/2} \gg u \sqrt{u} \log y .\]
Il vient
\[
I_4 = \alpha q_0^{-\alpha} \prod_{p | q_0}{\left( 1 - \frac{p^{\alpha} - 1}{p - 1} \right)} \cPhi(\beta x, \alpha) \Psi(x, y)
 + O_{\delta, \Phi} \left( \frac{\Psi(x, y) q_0^{-\alpha + \delta}}{(1 + |\beta| x) u}\right)
.\]

\bigskip

En combinant les estimations obtenues pour~$I_j$ ($1 \leq j \leq 7)$, on obtient la formule voulue.

\end{proof}

\subsection{Estimation de \texorpdfstring{$N(A, B, C, y ; \Phi)$}{N(A, B, C, y ; Phi)}}

On applique à présent la méthode du cercle comme dans~\cite{SoundLaga2011} en utilisant l'estimation de la proposition~\ref{estim_m}.
Pour toutes les valeurs de~$A, B, C$ telles que~$2/3 < \alpha_A, \alpha_B, \alpha_C < 1$, on définit les quantités suivantes
\begin{multline}\label{def_sigma0_multi}
{\mathfrak S}_{0} = {\mathfrak S}_{0}(\Phi ; A, B, C, y) \\ := \alpha_A \alpha_B \alpha_C \int_{0}^{\infty} {
        \int_{0}^{\infty} { \Phi(t_1) \Phi(t_2) \Phi(\lambda_1 t_1+\lambda_2 t_2) t_1^{\alpha_A-1} t_2^{\alpha_B-1} \left(\frac{A}{C} t_1 + \frac{B}{C} t_2\right)^{\alpha_C - 1} \dd t_1} \dd t_2}
\end{multline}
\[ {\mathfrak S}_{1} = {\mathfrak S}_{1}(A, B, C, y) := \prod_{p} {
        \left( 1 + \frac{(p-1) (p - p^{\alpha_A}) (p - p^{\alpha_B}) (p - p^{\alpha_C})}{p(p^{\alpha_A + \alpha_B + \alpha_C - 1} - 1)(p-1)^3} \right) } .\]

\begin{theoreme}\label{thm_NABC}
Il existe une constante absolue $c_0>0$ telle que pour tout~$\ee>0$ et~$\Phi$ une fonction de classe $\mathcal{C}^{\infty}$ à support compact inclus dans $]0, \infty [$,
il existe un réel~$\eta = \eta(\ee) >0$ qui tend vers $0$ avec $\ee$ tel que pour tous~$(A, y)$, $(B, y)$ et~$(C, y)$ dans le domaine $\cD^*(4+\eta, c_0)$ avec $C^\ee \leq A \leq B \leq C$ on ait
\begin{equation}\label{eqv_pcp_ABC}
\begin{split}
N(A, B, C, y ; \Phi) =&\ {\mathfrak S}_{0} {\mathfrak S}_{1}
                 \frac{\Psi(A, y)\Psi(B, y)\Psi(C, y)}{C} \\
				&\ + O_{\ee, \Phi}\left(\frac{\Psi(A, y)\Psi(B, y)\Psi(C, y)}{u_{A} C}
                + C^{3/4+\ee}\sqrt{\Psi(A, y) \Psi(C, y)} \right)
\end{split}
\end{equation}
\end{theoreme}

Avant de prouver ce théorème, on montre qu'il implique l'estimation~\eqref{eqv_pcp}.
\begin{proof}[Démonstration de la première assertion du Théorème~\ref{th1}]
Soit~$\ee>0$. Le Théorème~\ref{thm_NABC} dans le cas $A = B = C$ assure l'existence d'une constante absolue $c_0>0$ et d'un réel~$\eta = \eta(\ee)$ tel que pour tout~$(x, y) \in \cD^*(4 + \eta, c_0)$, on ait
\[
N(x, y ; \Phi) = {\mathfrak S}_{0}(\Phi ; \alpha) {\mathfrak S}_{1}(\alpha)
                 \frac{\Psi(x, y)^3}{x} \\
				+ O_{\ee, \Phi} \left(\frac{\Psi(x, y)^3}{u x}
                + x^{3/4+\ee}\Psi(x, y) \right)
.\]
On choisit $\ee \leq 1/48$ et tel que~$4 + \eta(\ee) \leq 8$.
On vérifie alors que pour~$(x, y) \in \cD(8 + 384\ee, c_0)$, on a
\begin{equation}\label{estim_erreur2_erreur1} x^{3/4+\ee}\Psi(x, y) \ll_\ee \frac{\Psi(x, y)^3}{x^{1+\ee}} \end{equation}
ce qui implique l'estimation~\eqref{eqv_pcp} en vertu de~$u \ll_\ee x^\ee$.
\end{proof}


\begin{proof}[Démonstration du Théorème~\ref{thm_NABC}]
Ce qui suit reprend en grande partie la démonstration du théorème~2.1 de~\cite{SoundLaga2011}.
Soit~$c_0$ la constante absolue donnée par la Proposition~\ref{estim_m}. On se place dans les hypothèses de l'énoncé.
On suppose que~$C \ll_\Phi B$ : dans le cas contraire, le membre de gauche et le terme principal du membre de droite de~\eqref{eqv_pcp_ABC} sont tous les deux nuls.
On désigne par $x$ une quantité générique telle que~$(x, y) \in \cD^*(4, c_0)$, c'est-à-dire qui puisse jouer le rôle de~$A$, $B$ et~$C$.
Enfin, on note que l'hypothèse $C^\ee \leq A \leq B \leq C$ implique~$\alpha_A = \alpha_B + o_\ee(1) = \alpha_C + o_\ee(1)$ lorsque~$A$, $B$ et~$C$ tendent vers l'infini.

On définit les arcs majeurs comme dans~\cite{SoundLaga2011}. Soit~$\delta > 0$. Pour tout~$q \leq C^{1/4}$ et~$a$ premier avec $q$ avec $0 \leq a < q$, on définit $\mathfrak M(a, q)$
pour~$q>1$ comme l'ensemble des $\vartheta \in [0, 1]$ tels que~$|\vartheta - a/q| \leq C^{\delta-1}$, et pour~$q=1$ comme l'ensemble $[0, C^{\delta-1}] \cup [1-C^{\delta-1}, 1]$.
Un réel de l'intervalle $[0, 1]$ appartient à au plus un tel ensemble. On note $\mathfrak M$ la réunion de tous les $\mathfrak M(a, q)$ pour~$q\geq 1$ et~$0 \leq a < q, (a, q)=1$,
et~$\mathfrak m = [0, 1] \smallsetminus \mathfrak M$.

Ainsi qu'il est montré dans la preuve du théorème~2.1 de~\cite{SoundLaga2011}, on a $E_\Phi(C, y ; \vartheta) \ll_{\ee, \Phi} C^{3/4+\ee}$ pour tout~$\vartheta \in \mathfrak m$.
On a donc par l'inégalité de Cauchy-Schwarz
\begin{align*}
\int_{\mathfrak m}{E_\Phi( A, y ; \vartheta) E_\Phi( B, y ; \vartheta) }&{ E_\Phi( C, y ; -\vartheta) \dd \vartheta} \\
& \ll_{\ee, \Phi} C^{3/4+\ee} \int_{0}^{1}{|E_\Phi(A, y ; \vartheta)| |E_\Phi(B, y ; -\vartheta)| \dd \vartheta} \\
& \ll C^{3/4+\ee} \sqrt{\Psi(A, y) \Psi(B, y)}
. \end{align*}

On examine maintenant les arcs majeurs. En utilisant trois fois la Proposition~\ref{SoundLaga2011_prop_6_1}, on obtient pour
$\vartheta \in \mathfrak M(a, q)$ avec $q\leq C^{1/4}$ et~$\vartheta = a/q + \beta$,
\begin{multline}\label{estim_E3_M3}
E_\Phi(A, y ; \vartheta) E_\Phi(B, y ; \vartheta) E_\Phi(C, y ; -\vartheta) = 
M_\Phi(A, y ; q, \beta) M_\Phi(B, y ; q, \beta) M_\Phi(C, y ; q, -\beta) \\
 + O_{\ee, \Phi}\left( B^{3/4+\ee/2} |E_\Phi(A, y ; \vartheta)| |E_\Phi(C, y ; \vartheta)| +
           A^{3/4+\ee/2} |M_\Phi(B, y ; q, \beta)|  |E_\Phi(C, y ; \vartheta)| + \right. \\
   \left.  C^{3/4+\ee/2} |M_\Phi(B, y ; q, \beta)|  |M_\Phi(A, y ; q, \beta)| \right)
.\end{multline}
Avant d'intégrer ceci pour~$|\beta| \leq C^{\delta-1}$ et sommer pour~$q \leq C^{1/4}$, on montre la majoration suivante
\begin{equation}\label{majo_integ_m}
I = I(x, y ; C, q_0) := \int_{-C^{\delta-1}}^{C^{\delta-1}}{ |M_\Phi(x, y ; q_0, \beta)|^2 \dd \beta}  \ll_{\ee, \Phi} q_0^{-1-\alpha+\ee/3} (\log x) \Psi(x, y)
\end{equation}
où~$P^+(q_0) \leq y$. Par définition de~$M_\Phi(x, y ; q_0, \beta)$ on a 
\begin{multline*}
 I = 2\sum_{P^+(n) \leq y}{\frac{1}{\varphi\left(\frac{q_0}{(q_0, n)}\right)^2}\Phi\left(\frac{n}{x}\right)^2 C^{\delta-1}} \\
   + 2 \Re \Bigg( \sum_{P^+(n) \leq y}{\sum_{\substack{m \geq n+1 \\ P^+(m) \leq y}}{
    \frac{1}{\varphi\left(\frac{q_0}{(q_0, n)}\right)}\frac{1}{\varphi\left(\frac{q_0}{(q_0, m)}\right)}
    \Phi\left(\frac{n}{x}\right)\Phi\left(\frac{m}{x}\right)
    h(m-n)}} \Bigg)
\end{multline*}
où~$h(r) = (\e(r C^{\delta-1}) - \e(-r C^{\delta-1}))/r$. On a pour tout~$r\geq 1$, $h(r) \ll C^{\delta-1}/(1+r C^{\delta-1})$.
On~note $I_1$ et~$I_2$ les deux sommes de cette estimation.
On a
\begin{align*}
I_1&\ \ll C^{\delta-1} \sum_{d | q_0}{\frac{1}{\varphi(q_0/d)^2} \sum_{P^+(n') \leq y}{\Phi(n' d / x)^2}} \\
 & \ll_\Phi C^{\delta-1} \sum_{d | q_0}{\frac{1}{d^\alpha \varphi(q_0/d)^2} \Psi(x, y)}
 \ll C^{\delta-1} q_0^{-\alpha} \Psi(x, y)
.\end{align*}
On note en effet que si le support de~$\Phi$ est inclus dans $[0, K]$
\[ \sum_{P^+(n') \leq y}{\Phi(n' d / x)} \ll_\Phi \Psi(K x/d, y) \ll_\Phi \Psi(x, y)/d^\alpha .\]
On écrit $I_2$ en divisant la somme en deux selon la taille de~$m-n$ par rapport à $C^{1-\delta}$ :
\begin{multline*}
 I_2 \ll C^{\delta-1}\sum_{P^+(n) \leq y}{\frac{1}{\varphi\left(\frac{q_0}{(q_0, n)}\right)}\Phi\left(\frac{n}{x}\right)
     \sum_{\substack{n+1 \leq m \leq n+C^{1-\delta} \\ P^+(m) \leq y}}{
    \frac{1}{\varphi\left(\frac{q_0}{(q_0, m)}\right)}
    \Phi\left(\frac{m}{x}\right) }} \\
   + \sum_{P^+(n) \leq y}{\frac{1}{\varphi\left(\frac{q_0}{(q_0, n)}\right)}\Phi\left(\frac{n}{x}\right)
     \sum_{\substack{m > n + C^{1-\delta} \\ P^+(m) \leq y}}{
    \frac{1}{\varphi\left(\frac{q_0}{(q_0, m)}\right)} \Phi\left(\frac{m}{x}\right) \frac{1}{m-n} }}
.\end{multline*}
On note $I_{2,1}$ et~$I_{2,2}$ les sommes qui apparaissent au membre de droite.
On a d'une part
\begin{align*}
I_{2,1} &\ \leq \sum_{P^+(n) \leq y}{\frac{1}{\varphi\left(\frac{q_0}{(q_0, n)}\right)}\Phi\left(\frac{n}{x}\right)
      \sum_{d | q_0}{\frac{1}{\varphi(q_0/d)} C^{\delta-1}
      \sum_{\frac{n+1}{d} \leq m \leq \frac{n+C^{1-\delta}}{d}}{\Phi\left(\frac{m' d}{x}\right) }}} \\
&  \ll  \sum_{P^+(n) \leq y}{\frac{1}{\varphi\left(\frac{q_0}{(q_0, n)}\right)}\Phi\left(\frac{n}{x}\right)
      \sum_{d | q_0}{\frac{1}{d \varphi(q_0/d)} }} \\
&  \ll_\ee q_0^{-1+\ee/4} \sum_{P^+(n) \leq y}{\frac{1}{\varphi\left(\frac{q_0}{(q_0, n)}\right)}\Phi\left(\frac{n}{x}\right)} \\
&  \leq q_0^{-1+\ee/4} \sum_{d | q_0}{\frac{1}{\varphi(q_0/d)} \sum_{P^+(n') \leq y}{\Phi\left(\frac{n' d}{x}\right)} } \\
&  \ll_\Phi q_0^{-1+\ee/4} \Psi(x, y) \sum_{d | q_0}{\frac{1}{d^\alpha \varphi(q_0/d)} } \\
&  \ll_\ee q_0^{-1-\alpha+\ee/3} \Psi(x, y)
\end{align*}
et d'autre part, si le support de~$\Phi$ est inclus dans $[0, K]$,
\begin{align*}
I_{2,2} &\ \ll_\Phi \sum_{P^+(n) \leq y}{\frac{1}{\varphi\left(\frac{q_0}{(q_0, n)}\right)}\Phi\left(\frac{n}{x}\right)
     \sum_{n + C^{1-\delta} \leq m \leq K x}{ \frac{1}{\varphi\left(\frac{q_0}{(q_0, m)}\right)} \frac{1}{m-n} }} \\
& \leq \sum_{P^+(n) \leq y}{\frac{1}{\varphi\left(\frac{q_0}{(q_0, n)}\right)}\Phi\left(\frac{n}{x}\right)
     \sum_{d | q_0}{ \frac{1}{\varphi(q_0/d)} \sum_{\frac{n + C^{1-\delta}}{d} \leq m' \leq \frac{K x}{d}}{ \frac{1}{m'd-n} }} }
.\end{align*}
La somme en~$m'$ peut être majorée par
\[ \int_{\frac{n+C^{1-\delta}}{d}-1}^{\frac{K x}{d}}{ \frac{\dd t}{d t - n} } = \frac{1}{d} \int_{C^{1-\delta} - d}^{K x - n}{\frac{\dd t}{t}} \ll_\Phi \frac{1}{d} \log x \]
en vertu de~$d \leq q_0 \leq C^{1/4} = o(C^{1-\delta})$ pour~$\delta$ assez petit.
Il vient
\begin{align*}
I_{2,2} &\ \ll_\Phi \log x \sum_{P^+(n) \leq y}{\frac{1}{\varphi\left(\frac{q_0}{(q_0, n)}\right)}\Phi\left(\frac{n}{x}\right)
     \sum_{d | q_0}{ \frac{1}{d \varphi(q_0/d)} }} \\
& \ll_\ee q_0^{-1+\ee/4} (\log x) \sum_{P^+(n) \leq y}{\frac{1}{\varphi\left(\frac{q_0}{(q_0, n)}\right)}\Phi\left(\frac{n}{x}\right) } \\
& \ll_\Phi q_0^{-1+\ee/4} (\log x) \Psi(x, y) \sum_{d | q_0}{ \frac{1}{d^\alpha \varphi(q_0/d)} } \\
& \ll_\ee q_0^{-1-\alpha+\ee/3} (\log x) \Psi(x, y)
.\end{align*}
On a donc, en tenant compte de~$q_0 \leq C^{1/4}$,
\[ I \ll_{\ee, \Phi} q_0^{-1-\alpha + \ee/3} (\log x) \Psi(x, y) .\]
À propos du terme en~$q_0$, on note qu'en sommant sur $q \leq C^{1/4}$ on a
\begin{equation}\label{majo_somme_q}
 \sum_{q \leq C^{1/4}}{\frac{\varphi(q_0)}{\varphi(q_1)} q_0^{-1-\alpha+\ee/3} }
  \leq \Bigg( \sum_{q_0 \in S(C^{1/4},\ y)}{\varphi(q_0) q_0^{-1-\alpha+\ee/3}} \Bigg)
      \Big( \sum_{q_1 \leq C^{1/4}}{\varphi(q_1)^{-1}} \Big)
 \ll_\ee C^{\ee/3}
\end{equation}
où on écrit de manière unique~$q = q_0 q_1$ avec $P^+(q_0) \leq y$ et~$P^-(q_1)>y$. La somme sur $q_0$ est traitée par une sommation d'Abel, grâce à l'estimation~\eqref{psi_x_alpha}.

On intègre maintenant les termes d'erreur de~\eqref{estim_E3_M3} sur les arcs majeurs.
On a un premier terme
\[ \ll_{\ee, \Phi} B^{3/4+\ee/2} \int_{\mathfrak M}{|E_\Phi(A, y ; \vartheta)| |E_\Phi(C, y ; \vartheta)| \dd \vartheta} \leq B^{3/4+\ee/3} \sqrt{\Psi(A, y) \Psi(C, y)} .\]
En utilisant $|M_\Phi(B, y ; q, \beta)| |E_\Phi(C, y ; \vartheta)| \ll |M_\Phi(B, y ; q, \beta)|^2 + |E_\Phi(C, y ; \vartheta)|^2 $, on voit que le deuxième terme d'erreur de~\eqref{estim_E3_M3} donne après intégration un terme
\[ \ll_{\ee, \Phi} A^{3/4+\ee/2} \Big\{ \Psi(C, y) + \sum_{q \leq C^{1/4}}{\varphi(q) \int_{-C^{1-\delta}}^{C^{1-\delta}}{ |M_\Phi(B, y ; q, \beta)|^2 }} \Big\} .\]
Les majorations~\eqref{majo_integ_m} et~\eqref{majo_somme_q} montrent que ceci est
\[ \ll_{\ee, \Phi} A^{3/4+\ee/2} \left\{ \Psi(C, y) + (\log B) \Psi(B, y) C^{\ee/3} \right\} .\]
Le dernier terme d'erreur de~\eqref{estim_E3_M3} devient après sommation et intégration un terme d'erreur
\[ \ll_{\ee, \Phi} C^{3/4+\ee/2} \sum_{q \leq C^{1/4}}{ \frac{\varphi(q_0)}{\varphi(q_1)}
                       \int_{-C^{\delta-1}}^{C^{\delta-1}}{ |M_\Phi(A, y ; q_0, \beta)| |M_\Phi(B, y ; q_0, \beta)| \dd \beta} } \]
et une inégalité de Cauchy-Schwarz ainsi que les calculs faits précédemment montrent que celui-ci est
\[ \ll_{\ee, \Phi} C^{3/4+\ee} \sqrt{ \Psi(A, y) \Psi(B, y) } .\]
Les hypothèses $A \leq C$ et~$B \asymp_\Phi C$ que l'on a faites montrent que chacun de ces trois termes d'erreur est
\[ \ll_{\ee, \Phi} C^{3/4+\ee} \sqrt{ \Psi(A, y) \Psi(B, y) } \]
et on en déduit
\begin{equation}\label{estim_arcmaj_1}
\begin{split}
\int_{\mathfrak M}{E_\Phi(A, y ; \vartheta) }&{ E_\Phi(B, y ; \vartheta) E_\Phi(C, y ; -\vartheta) \dd \vartheta} \\
= &\ \sum_{q \leq C^{1/4}}{ \varphi(q) \int_{-C^{\delta-1}}^{C^{\delta-1}}{M_\Phi(A, y ; q, \beta) M_\Phi(B, y ; q, \beta) M_\Phi(C, y ; q, -\beta) \dd \beta}} \\
&\ + O_{\ee, \Phi} \left(C^{3/4+2\ee} \sqrt{\Psi(A, y) \Psi(B, y)}\right)
. \end{split}
\end{equation}

On applique le Lemme~\ref{estim_m} avec $R=C$, pour successivement $x=A$, $x=B$ et~$x=C$.
On~reporte les estimations ainsi obtenues pour~$M_\Phi(A, y ; q, \beta)$, $M_\Phi(B, y ; q, \beta)$ et~$M_\Phi(C, y ; q, \beta)$ dans~\eqref{estim_arcmaj_1}.
Une étude détaillée des termes d'erreur permet d'écrire, compte tenu de~$B \asymp_\Phi C$ et~$\alpha_A = \alpha_C + o(1)$,
\begin{equation}\label{estim_mmm}
\begin{split}
M_\Phi(A, y ; q, \beta) M_\Phi(B, y ; q, \beta) &\ M_\Phi(C, y ; q, -\beta) \\
= \frac{\mu(q_1)}{\varphi(q_1)^3 } \widetilde{M}_\Phi(A, &\ y ; q_0, \beta) \widetilde{M}_\Phi(B, y ; q_0, \beta) \widetilde{M}_\Phi(C, y ; q_0, \beta) \\
+ O_{\delta, \Phi}\Bigg( & \ \frac{1}{\varphi(q_1)^3} \frac{q_0^{-\alpha_A-\alpha_B-\alpha_C+3\delta} \Psi(A, y)\Psi(B, y)\Psi(C, y)}{u_A (1+|\beta| A)(1+|\beta| B)(1+|\beta| C)} \Bigg) \\
+ O_{\delta, \Phi}\Big( & \ \frac{1}{\varphi(q_1)^3} \Big\{ q_0^{3/2} A C^{-1/2 + 4\delta} + \Psi(A, y) q_0^{-1-\alpha_C + \delta} C^{3\delta} \\
&\ + \Psi(A, y) \Psi(C, y) q_0^{-1/2-2\alpha_C+2\delta}C^{-1/2+2\delta} \Big\} \Big)
.\end{split}
\end{equation}

On suppose $(C, y) \in \cD(4+\eta, c_0)$ pour un certain $\eta>0$ qui garantisse $\alpha_C - 3/4 \geq 3\delta$.
On~peut vérifier que le choix $\eta = 97\delta$ est valable dès que~$\delta \leq 1/24$ pour~$x$ et~$y$ assez grands.
On~reporte l'estimation~\eqref{estim_mmm} dans l'estimation~\eqref{estim_arcmaj_1}. On considère tout d'abord le deuxième terme d'erreur.
Après intégration et sommation, on obtient un terme
\[ \ll_{\delta, \Phi} A C^{-3/8 + \delta} + \Psi(A, y) C^{1/4 - \alpha_C/4 + 5 \delta} + \Psi(A, y) \Psi(C, y) C ^{-1/2 + 3\delta} .\]
On vérifie alors que ceci est $\ll_{\ee, \Phi} C^{3/4 + \ee} \sqrt{\Psi(A, y) \Psi(B, y)}$ quitte à prendre $\delta$ assez petit en fonction de~$\varepsilon$.

On considère ensuite le premier terme d'erreur de~\eqref{estim_mmm}.
Après intégration et sommation, on obtient un terme
\begin{align*}
\ll_{\delta, \Phi} &\int_{-C^{\delta-1}}^{C^{\delta-1}}{\frac{\dd \beta}{(1+|\beta| A)(1+|\beta| B)(1+|\beta| C)}}  \frac{\Psi(A, y)\Psi(B, y)\Psi(C, y)}{u_A} \\ 
  & \times \sum_{q_0 \in S(C^{1/4}, y)}{\frac{\varphi(q_0)}{q_0^{\alpha_A+\alpha_B+\alpha_C-3\delta}} \sum_{\substack{q_1 \leq C^{1/4}/q_0 \\ P^-(q_1)>y}}{
  \frac{1}{\varphi(q_1)^2} }}
. \end{align*}
La somme en~$q_1$ est trivialement bornée.
Étant donnée notre hypothèse $\alpha_C - 3/4 \geq 3\delta$, la somme en~$q_0$ est majorée par
\[ \sum_{q_0 \geq 1}{\varphi(q_0)q_0^{-\alpha_A-\alpha_B-\alpha_C+3\delta}} \ll 1 .\]
Enfin, l'intégrale se majore comme suit :
\[ \int_{-C^{\delta-1}}^{C^{\delta-1}}{\frac{\dd \beta}{(1+|\beta| A)(1+|\beta| B)(1+|\beta| C)}}
   \ll \frac{1}{C} \int_{0}^{C^{\delta}}{\frac{\dd \xi}{(1+ \frac{A}{C} \xi )(1+ \frac{B}{C}\xi)(1+ \xi)}}
   \ll_\Phi \frac{1}{C} \]
la dernière majoration étant valable en vertu de~$B/C \gg_\Phi 1$.
Le premier terme d'erreur de~\eqref{estim_mmm} devient donc lorsqu'on le reporte dans~\eqref{estim_arcmaj_1} un terme
\[ \ll_{\delta, \Phi} \frac{\Psi(A, y)\Psi(B, y)\Psi(C, y)}{u_A C} .\]

On reporte maintenant le terme principal de~\eqref{estim_mmm} dans l'estimation~\eqref{estim_arcmaj_1}. Une inversion de Fourier fournit
\begin{align*}
\int_{-C^{\delta-1}}^{C^{\delta-1}}{\cPhi(\beta A, \alpha_A)} & {\cPhi(\beta B, \alpha_B) \cPhi(-\beta C, \alpha_C) \dd \beta} \\
= &\ \frac{1}{C} \int_{-C^{\delta}}^{C^{\delta}}{\cPhi(\xi A/C, \alpha_A) \cPhi(\xi B/C, \alpha_B) \cPhi(-\xi, \alpha_C) \dd \xi} \\
= &\ \frac{1}{C} \int_{-\infty}^{\infty}{\cPhi(\xi A/C, \alpha_A) \cPhi(\xi B/C, \alpha_B) \cPhi(-\xi, \alpha_C) \dd \xi}  + O_\Phi(C^{-1-\delta})\\
= &\ \frac{1}{C} {\mathfrak S}_{0}(\Phi ; A, B, C, y) + O_\Phi(C^{-1-\delta})
\end{align*}
Quant à la somme en~$q$, on vérifie que l'hypothèse $\alpha_C - 3/4 \geq 3\delta$ implique
\[ \frac{1}{y^{\alpha_A+\alpha_B+\alpha_C-2} \log y} \ll \frac{1}{u} .\]
On a alors
\begin{align*}
& \sum_{q_0 \in S(C^{1/4}, y)}{\sum_{\substack{q_1 \leq C^{1/4}/q_0 \\ P^-(q_1)>y}}{
        \frac{\mu(q_1)}{\varphi(q_1)^2} \frac{\alpha_A \alpha_B \alpha_C}{q_0^{\alpha_A + \alpha_B + \alpha_C}}
        \prod_{p|q_0}{\left( 1-\frac{p^{\alpha_A}-1}{p-1} \right) \left( 1-\frac{p^{\alpha_B}-1}{p-1} \right) \left( 1-\frac{p^{\alpha_C}-1}{p-1} \right)} }} \\
= &\ \sum_{P^+(q_0) \leq y}{\sum_{\substack{q_1 \geq 1 \\ P^-(q_1)>y}}{\frac{\mu(q_1)}{\varphi(q_1)^2} \frac{\alpha_A \alpha_B \alpha_C}{q_0^{\alpha_A + \alpha_B + \alpha_C}}
  \prod_{p|q_0}{\left( 1-\frac{p^{\alpha_A}-1}{p-1} \right) \left( 1-\frac{p^{\alpha_B}-1}{p-1} \right) \left( 1-\frac{p^{\alpha_C}-1}{p-1} \right)} }} \\
  & + O_\delta\left(C^{1/4(2+\delta-\alpha_A-\alpha_B-\alpha_C)} \right)
\end{align*}
où l'hypothèse $\alpha_C - 3/4 \geq 3\delta$ implique que l'exposant dans le terme d'erreur est $\leq -1/16$.
La somme en~$q_1$ vérifie
\[ \sum_{\substack{q_1 \geq 1 \\ P^-(q_1)>y}}{\frac{\mu(q_1)}{\varphi(q_1)^2}} = 1 + O\left(\frac{1}{y \log y}\right) \]
et pour la somme en~$q_0$ on a
\begin{align*}
\sum_{P^+(q_0) \leq y}{
  q_0^{-\alpha_A - \alpha_B - \alpha_C} }&{ \prod_{p|q_0}{\left( 1-\frac{p^{\alpha_A}-1}{p-1} \right) \left( 1-\frac{p^{\alpha_B}-1}{p-1} \right) \left( 1-\frac{p^{\alpha_C}-1}{p-1} \right)}
} \\
& = \prod_{p\leq y}{ \left( 1 + \frac{p-1}{p(p^{\alpha_A + \alpha_B + \alpha_C - 1} - 1)}
        \left( \frac{p-p^{\alpha_A}}{p-1} \right) \left( \frac{p-p^{\alpha_B}}{p-1} \right) \left( \frac{p-p^{\alpha_C}}{p-1} \right)
        \right) } \\
& = {\mathfrak S}_{1}(A, B, C, y) + O\left(\frac{1}{y^{\alpha_A+\alpha_B+\alpha_C-2} \log y}\right)
\end{align*}
On utilise ${\mathfrak S}_{0}(\Phi ; A, B, C, y) \ll 1$ et~${\mathfrak S}_{1}(A, B, C, y) \ll 1$
et on choisit $\delta$ suffisamment petit, en fonction de~$\ee$, afin d'obtenir
\begin{align*}
& \int_{\mathfrak M}{E_\Phi(A, y ; \vartheta)  E_\Phi(B, y ; \vartheta) E_\Phi(C, y ; -\vartheta) \dd \vartheta} \\
= &\ {\mathfrak S}_{0} {\mathfrak S}_{1} \frac{\Psi(A, y)\Psi(B, y)\Psi(C, y)}{C} + O_{\ee, \Phi}\left(C^{3/4+\ee} \sqrt{\Psi(A, y) \Psi(B, y)} + \frac{\Psi(A, y)\Psi(B, y)\Psi(C, y)}{u_A C} \right)
\end{align*}
et donc
\begin{align*}
N(A, B, C, y ; \Phi)
= &\ {\mathfrak S}_{0} {\mathfrak S}_{1} \frac{\Psi(A, y)\Psi(B, y)\Psi(C, y)}{C} \\
& + O_{\ee, \Phi}\left(C^{3/4+\ee} \sqrt{\Psi(A, y) \Psi(B, y)} + \frac{\Psi(A, y)\Psi(B, y)\Psi(C, y)}{u_A C} \right)
\end{align*}
ce qui achève la démonstration du théorème.

\end{proof}

\section{Solutions primitives pondérées et solutions non pondérées}

\subsection{Étude des solutions primitives pondérées}

On montre dans cette partie la deuxième assertion du Théorème~\ref{th1}. Soit~$c_0>0$ la constante absolue donnée par le Théorème~\ref{thm_NABC}. Soit~$\ee>0$ et~$(x, y) \in \cD(8+\ee, c_0)$.
Une inversion de Möbius fournit
\[
N^*(x, y ; \Phi) = \sum_{P^+(d) \leq y}{\mu(d) N\left(\frac{x}{d}, y ; \Phi\right)}
.\]

On pose $D_0 = x^{\delta_0}$ avec
\[ \delta_0 = \frac{1+\ee-\alpha}{2\alpha-1} \]
de sorte que~$D_0^{1-2\alpha} \ll \Psi(x, y) x^{-1-\ee/2}$.
Dans le domaine $\cD(8+\ee, c_0)$, on a $\alpha - 2/3 \gg 1$ donc pour~$\ee$ suffisamment petit les inégalités
$0 < \delta_0 < 1$, $1 \ll \delta_0$ et~$ 1 \ll 1 - \delta_0$ sont valables. Lorsque~$d > D_0$, on utilise la majoration triviale
\[ N(x/d, y ; \Phi) \ll_\Phi \Psi(K x/d, y)^2 \ll_\Phi d^{-2\alpha} \Psi(x, y)^2 \]
où~$K$ est tel que le support de~$\Phi$ est inclus dans $]0, K]$. On a donc
\begin{align*}
\sum_{\substack{P^+(d) \leq y \\ d > D_0}}{\mu(d) N\left(\frac{x}{d}, y ; \Phi\right)}
& \ll_\Phi \Psi(x, y)^2 \sum_{D_0 < d \leq K x}{d^{-2\alpha}} \\
& \ll \Psi(x, y)^2 D_0^{1-2\alpha}  \ll \frac{\Psi(x, y)^3}{x^{1+\ee/2}}
\end{align*}
qui est bien inclus dans le terme d'erreur de~\eqref{eqv_pcp_prim}.
Lorsque~$d \leq D_0$, on a $\log(x/d) \gg \log x$ ce qui implique que~$x/d$ est dans le domaine $\cD(8+\ee, c_0)$ quitte à prendre $c_0$ suffisamment petit. 
On peut donc utiliser l'estimation~\eqref{eqv_pcp} et écrire
\begin{align*}
N^*(x, y ; \Phi) = &\ \sum_{\substack{P^+(d) \leq y \\ d \leq D_0}}
                 {\mu(d){\mathfrak S}_{0}(\Phi, \alpha_{x/d}) {\mathfrak S}_{1}(\alpha_{x/d})
                  \frac{\Psi(x/d, y)^3}{x/d}\left\{1 + O_{\ee, \Phi}\left(\frac{1}{u_{x/d}}\right)\right\}} \\
                 &\ + O_\Phi\left(\frac{\Psi(x, y)^3}{x^{1+\ee}} \right)
.\end{align*}
On note que~$u_{x/d} \asymp u$ pour~$d \leq D_0$, ainsi que~$\sum_{d \geq 1}{d^{1-3\alpha}} \ll 1$. On a donc
\[
N^*(x, y ; \Phi) = \sum_{\substack{P^+(d) \leq y \\ d \leq D_0}}
                 {\mu(d){\mathfrak S}_{0}(\Phi, \alpha_{x/d}) {\mathfrak S}_{1}(\alpha_{x/d})
                  \frac{\Psi(x/d, y)^3}{x/d} }
                 + O_{\ee, \Phi}\left(\frac{\Psi(x, y)^3}{u x}\right)
.\]

Soit~$D_1 = u^{1/(3\alpha-2)}$. L'inégalité~$\alpha - 2/3 \gg 1$ fournit
\begin{equation}\label{somme_reste_d}
\sum_{d > D_1}{d^{1-3\alpha}} \ll \frac{1}{u}
\end{equation}
et ainsi 
\[ 
N^*(x, y ; \Phi) = \sum_{\substack{P^+(d) \leq y \\ d \leq D_1}}
                  {\mu(d){\mathfrak S}_{0}(\Phi, \alpha_{x/d}) {\mathfrak S}_{1}(\alpha_{x/d})
                  \frac{\Psi(x/d, y)^3}{x/d} }
                 + O_{\ee, \Phi}\left(\frac{\Psi(x, y)^3}{u x}\right)
.\]
Il est montré dans~\cite{TeneHild86} (équation 6.6) que l'on a
\[ \alpha_{x/d} - \alpha \ll \frac{\log d}{u (\log y)^2} .\]
Dans le domaine $\cD(8+\ee, c_0)$ et lorsque~$d \leq D_1$
\[ \frac{\log d}{u (\log y)^2} \ll \frac{\log u}{u (\log y)^2} \ll \frac{1}{u} .\]
On vérifie d'une part par convergence dominée que~$(\partial{\mathfrak S}_{0} / \partial\alpha) (\Phi, \alpha) \ll 1$,
d'autre part en considérant la dérivée logarithmique que~${\mathfrak S}_{1}'(\alpha) \ll 1$.
On a donc
\[ {\mathfrak S}_{0}(\Phi, \alpha_{x/d}) = {\mathfrak S}_{0}(\Phi, \alpha) + O\left(\frac{\log d}{u (\log y)^2}\right) \]
\[ {\mathfrak S}_{1}(\alpha_{x/d}) = {\mathfrak S}_{1}(\alpha) + O\left(\frac{\log d}{u (\log y)^2}\right) .\]
Il vient
\[ 
N^*(x, y ; \Phi) = {\mathfrak S}_{0}(\Phi, \alpha) {\mathfrak S}_{1}(\alpha)
                   \sum_{\substack{P^+(d) \leq y \\ d \leq D_1}}{\mu(d) \frac{\Psi(x/d, y)^3}{x/d} }
                 + O_{\ee, \Phi}\left(\frac{\Psi(x, y)^3}{u x}\right)
.\]
On a ensuite $t = (\log d)/\log y \ll (\log u)/\log y \ll 1$. L'estimation du Lemme~\ref{estim_psi_local} fournit donc
\[ \Psi\left(\frac{x}{d}, y\right) = \left\{1 + O\left(\frac{1}{u}\right)\right\}\frac{\Psi(x, d)}{d^\alpha} \]
et ainsi
\begin{align*}
N^*(x, y ; \Phi) & = {\mathfrak S}_{0}(\Phi, \alpha) {\mathfrak S}_{1}(\alpha) \frac{\Psi(x, y)^3}{x}
                   \sum_{\substack{P^+(d) \leq y \\ d \leq D_1}}{\frac{\mu(d)}{d^{3\alpha-1}} }
                 + O_{\ee, \Phi}\left(\frac{\Psi(x, y)^3}{u x}\right) \\
& = \frac{{\mathfrak S}_{0}(\Phi, \alpha) {\mathfrak S}_{1}(\alpha)}{\zeta(3\alpha-1, y)} \frac{\Psi(x, y)^3}{x}
                 + O_{\ee, \Phi}\left(\frac{\Psi(x, y)^3}{u x}\right)
\end{align*}
en vertu encore une fois de~\eqref{somme_reste_d}.
Ceci montre l'estimation~\eqref{eqv_pcp_prim} et achève la démonstration du Théorème~\ref{th1}.

\subsection{Étude des solutions non pondérées}

On montre dans cette section l'estimation~\eqref{eqv_nonpond}. L'estimation~\eqref{eqv_nonpond_prim} s'ensuit par une méthode identique à celle de la section précédente, simplifiée
par le fait qu'on ne se préoccupe plus de la taille du terme d'erreur.
Lorsque~$(x, y)$ vérifie les hypothèses du théorème~\ref{BGthm}, on a $(\log u)/\log y \ll 1/u$ ainsi que~$\alpha = 1 + O((\log u) / \log y)$ ce qui implique
\[ {\mathfrak S}_{0}({\mathbf 1}_{]0,1]}, \alpha) {\mathfrak S}_{1}(\alpha) = \frac{1}{2} + O\left(\frac{\log u}{\log y}\right) \]
et l'estimation~\eqref{eqv_nonpond} découle donc du théorème~\ref{BGthm}.

On suppose donc que~$(x, y)$ ne vérifie pas les hypothèses du théorème~\ref{BGthm} pour par exemple $\ee=1/6$, en particulier, pour tout~$c_0$ fixé on a  $(\log y) H(u)^{-c_0} \leq 1$ pour $x$ et $y$ assez grands.
Il s'agit d'établir la borne supérieure de l'estimation~\eqref{eqv_nonpond}, puisque la borne inférieure est montrée dans~\cite{SoundLaga2011}.

Soit~$\ee>0$. On part de l'expression
\[
N(x, y) = \sum_{\substack{\kk \in \bfN^3 \\ k_1 \geq k_3 \\ k_2 \geq k_3}}{N(x2^{-k_1}, x2^{-k_2}, x2^{-k_3}, y ; {\mathbf 1}_{]1/2, 1]})}
 \]
où on a noté $\kk = (k_1, k_2, k_3)$.

Soit~$K_0$ tel que que~$2^{K_0} = x^{\delta_0}$, avec $\delta_0 = 1/\alpha - 1 + \ee$,
de sorte que~$2^{-\alpha K_0} = o(\Psi(x, y)x^{-1})$.
Alors par une majoration triviale et l'utilisation du Lemme~\ref{estim_psi_local} on a
\begin{align*}
\sum_{\substack{\kk \in \bfN^3 \\ \max(k_1, k_2) \geq K_0}}{N(x2^{-k_1}, x2^{-k_2}, x2^{-k_3}, y ; {\mathbf 1}_{]1/2, 1]})} 
 &\ \ll N(x 2^{-K_0}, x, x, y ; {\mathbf 1}_{]0, 1]}) \\
 &\ \ll 2^{-\alpha K_0} \Psi(x, y)^2 = o\left(\frac{\Psi(x, y)^3}{x}\right)
.\end{align*}
On a donc
\[
N(x, y) = \sum_{\substack{\kk \in \bfN^3 \\ K_0 \geq k_1 \geq k_3 \\ K_0 \geq k_2 \geq k_3}}{N(x2^{-k_1}, x2^{-k_2}, x2^{-k_3}, y ; {\mathbf 1}_{]1/2, 1]})}
 + o\left(\frac{\Psi(x, y)^3}{x}\right)
\]

On fixe $\delta>0$ et on choisit une fonction $\Phi$ de classe $\cC^\infty$ vérifiant
${\mathbf 1}_{]1/2, 1]} \leq \Phi \leq {\mathbf 1}_{](1-\delta)/2, 1+\delta]}$.
On a
\[
N(x, y) \leq \sum_{\substack{\kk \in \bfN^3 \\ K_0 \geq k_1 \geq k_3 \\ K_0 \geq k_2 \geq k_3}}{N(x2^{-k_1}, x2^{-k_2}, x2^{-k_3}, y ; \Phi)}
 + o\left(\frac{\Psi(x, y)^3}{x}\right)
\]
On note pour simplifier $A := x2^{-k_1}$, $B := x2^{-k_2}$, $C := x2^{-k_3}$.
On peut appliquer le Théorème~\ref{thm_NABC}. Soient~$c_0>0$ la constante absolue et~$\eta = \eta(\ee) >0$ le réel donnés par le théorème, et supposons $(x, y) \in \cD(4+\eta)$.
On a alors $\alpha - 1/2 \gg 1$ donc $(\log A)/(\log C) \gg 1$ pour tous les indices $k_1$, $k_3$ vérifiant $k_1 \leq K_0$ et~$k_3 \geq 0$.
On vérifie ensuite que quitte à diminuer la constante absolue $c_0$ dans la définition de~$\cD(4+\eta)$,
on a $(A, y), (B, y), (C, y) \in \cD(4+\eta)$ lorsque~$0 \leq k_1, k_2, k_3 \leq K_0$ : pour ces valeurs des indices,
l'estimation~\eqref{eqv_pcp_ABC} est donc valable.
En remarquant de plus que
\[ \sum_{\kk \in \bfN^3}{C^{3/4+\ee}\sqrt{\Psi(A, y) \Psi(B, y)}} \ll x^{3/4+\ee}\Psi(x, y) \]
on en déduit
\begin{align*}
N(x, y) \leq &\ \sum_{\substack{\kk \in \bfN^3 \\ K_0 \geq k_1 \geq k_3 \\ K_0 \geq k_2 \geq k_3}} {{\mathfrak S}_1(A, B, C, y)
 {\mathfrak S}_0(\Phi ; A, B, C, y) \frac{\Psi(A, y) \Psi(B, y) \Psi(C, y)}{C}} \\
 & + O_{\ee, \Phi}\left( x^{3/4+\ee}\Psi(x, y) \right)
   + o\left( \frac{\Psi(x, y)^3}{x} \right)
.\end{align*}

On choisit $\ee \leq 1/48$ et tel que~$4+\eta(\ee) \leq 8$.
On suppose $(x, y) \in \cD(8+384\ee)$, et on a alors de la même façon que dans la formule~\eqref{estim_erreur2_erreur1},
\[ x^{3/4+\ee}\Psi(x, y) = o\left( \frac{\Psi(x, y)^3}{x}\right) .\]

On pose $K_1 = \lfloor \log y / \log 2 \rfloor$ de sorte que~$2^{K_1} \asymp y$.
Pour les valeurs des indices sur lesquelles on somme, on a
\[ {\mathfrak S}_1 (A, B, C, y) {\mathfrak S}_0(\Phi ; A, B, C, y) \ll 1 .\]
La contribution des indices vérifiant $\max(k_1, k_2) \geq K_1$ est donc
\begin{align*}
& \ll \sum_{\substack{\kk \in \bfN^3 \\ k_2 \geq k_3 \\ k_1 \geq K_1}}{\frac{\Psi(x2^{-k_1}, y)\Psi(x2^{-k_2}, y)\Psi(x2^{-k_3}, y)}{x2^{-k_3}}} \\
& \ll \frac{\Psi(x, y)^3}{x} \sum_{\substack{\kk \in \bfN^3 \\ k_2 \geq k_3 \\ k_1 \geq K_1}}{2^{-\alpha_A k_1 - \alpha_B k_2 + (1-\alpha_C)k_3}} \\
& \ll \frac{\Psi(x, y)^3}{x} 2^{-\alpha K_1} \sum_{k_3 \geq 0}{2^{(1 - 2\alpha)k_3}} = o\left(\frac{\Psi(x, y)^3}{x}\right)
.\end{align*}
Ainsi
\begin{align*}
N(x, y) \leq &\ \sum_{\substack{\kk \in \bfN^3 \\ K_1 \geq k_1 \geq k_3 \\ K_1 \geq k_2 \geq k_3}}{{\mathfrak S}_1(A, B, C, y)
 {\mathfrak S}_0(\Phi ; A, B, C, y) \frac{\Psi(A, y) \Psi(B, y) \Psi(C, y)}{C}} \\
 & + o_{\ee, \Phi}\left(\frac{\Psi(x, y)^3}{x}\right)
.\end{align*}

En utilisant la définition~\eqref{def_sigma0_multi} on écrit
\begin{multline*} {\mathfrak S}_{0}(\Phi ; A, B, C, y) =  \alpha_A \alpha_B \alpha_C
 2^{\alpha_A k_1} 2^{\alpha_B k_2} 2^{(\alpha_C - 1)k_3} \times \\
 \int_{0}^{\infty} { \int_{0}^{\infty} { \Phi(v_1 2^{k_1}) \Phi(v_2 2^{k_2}) \Phi((v_1+v_2)2^{k_3})
                         v_1^{\alpha_A-1} v_2^{\alpha_B-1} (v_1 + v_2)^{\alpha_C - 1} \dd v_1} \dd v_2}
\end{multline*}
où l'on a effectué les changements de variables $t_1 \gets v_1 2^{k_1}$ et~$t_2 \gets v_2 2^{k_2}$.
Pour les valeurs que parcourent les indices $k_1$, $k_1$ et~$k_3$ on a
\[ \max( |\alpha_A - \alpha|, |\alpha_B - \alpha|, |\alpha_C - \alpha|) \ll 1/(\log x) .\]
En particulier on a
\[ {\mathfrak S}_1(A, B, C, y) \sim {\mathfrak S}_1(\alpha) .\]
Quant à la double intégrale, son intégrande est à support compact inclus dans $[0, 1+\delta]^2$, et pour~$v_1, v_2 \leq 1+\delta$ on a
\[ v_1^{\alpha_A-1} v_2^{\alpha_B-1} (v_1 + v_2)^{\alpha_C - 1} \sim \left(v_1 v_2 (v_1 + v_2)\right)^{\alpha - 1} .\]
On a donc
\begin{align*}
N(x, y) \leq &\ {\mathfrak S}_1(\alpha) \alpha³ \sum_{\substack{\kk \in \bfN^3 \\ K_1 \geq k_1 \geq k_3 \\ K_1 \geq k_2 \geq k_3}}
  {2^{\alpha_A k_1} 2^{\alpha_B k_2} 2^{(\alpha_C - 1)k_3} \frac{\Psi(A, y) \Psi(B, y) \Psi(C, y)}{C}} \\ &\ { \times
 \int_{0}^{\infty} { \int_{0}^{\infty} { \Phi(v_1 2^{k_1}) \Phi(v_2 2^{k_2}) \Phi((v_1+v_2)2^{k_3})
                         \left(v_1 v_2 (v_1 + v_2)\right)^{\alpha - 1} \dd v_1} \dd v_2}  }\\
 & + o_{\ee, \Phi}\left(\frac{\Psi(x, y)^3}{x}\right)
.\end{align*}
Le Lemme~\ref{estim_psi_local} avec l'inégalité $K_1 \ll \log y$ fournit uniformément lorsque~$0 \leq k_1, k_2, k_3 \leq K_1$ les équivalents
\[ \Psi(A, y) \sim 2^{-\alpha_A k_1} \Psi(x, y) \]
\[ \Psi(B, y) \sim 2^{-\alpha_B k_2} \Psi(x, y) \]
\[ \Psi(C, y) \sim 2^{-\alpha_C k_3} \Psi(x, y) \]
de sorte que
\begin{align*}
N(x, y) \leq &\ {\mathfrak S}_1(\alpha) \alpha³ \frac{\Psi(x, y)^3}{x} \\ &\ \times \sum_{\substack{\kk \in \bfN^3 \\ K_1 \geq k_1 \geq k_3 \\ K_1 \geq k_2 \geq k_3}}
 \int_{0}^{\infty} { \int_{0}^{\infty} { \Phi(v_1 2^{k_1}) \Phi(v_2 2^{k_2}) \Phi((v_1+v_2)2^{k_3})
                         \left(v_1 v_2 (v_1 + v_2)\right)^{\alpha - 1} \dd v_1} \dd v_2} \\
 & + o_{\ee, \Phi}\left(\frac{\Psi(x, y)^3}{x}\right)
.\end{align*}

Soit~$v \in {\mathbf R}$. De la majoration $\Phi \leq {\mathbf 1}_{](1-\delta)/2, 1+\delta]}$ on déduit, quitte à supposer $\delta < 1/2$,
\[ \sum_{k \geq 0}{\Phi(v 2^k)} \leq {\mathbf 1}_{]0, 1]}(v) + {\mathbf 1}_{I_{\delta}}(v) \]
où on a noté $I_{\delta} = \cup_{k \geq 0}{](1-\delta)2^{-k}, (1+\delta)2^{-k}]}$. Alors, en étendant la somme sur chaque indice à ${\mathbf N}$, on a
\begin{align*}
& \sum_{\kk \in \bfN^3}
 \int_{0}^{\infty} { \int_{0}^{\infty} { \Phi(v_1 2^{k_1}) \Phi(v_2 2^{k_2}) \Phi((v_1+v_2)2^{k_3})
                         \left(v_1 v_2 (v_1 + v_2)\right)^{\alpha - 1} \dd v_1} \dd v_2} \\
\leq &\ \int_{0}^{1} { \int_{0}^{1-t_2} { \left(t_1 t_2 (t_1 + t_2)\right)^{\alpha - 1} \dd t_1} \dd t_2} \\
& + O\left(\int_{I_{\delta}} { \int_{0}^{1} { \left(t_1 t_2 (t_1 + t_2)\right)^{\alpha - 1} \dd t_1} \dd t_2}
         + \int_{0}^{1} { \int_{0}^{1} { {\mathbf 1}_{I_{\delta}}(t_1+t_2) \left(t_1 t_2 (t_1 + t_2)\right)^{\alpha - 1} \dd t_1} \dd t_2} \right)
.\end{align*}
Par convergence dominée, le terme d'erreur tend vers $0$ avec $\delta$. On dispose donc d'une fonction $f(\delta)$ qui tend vers $0$
telle que
\[ N(x, y) \leq (1 + f(\delta)) {\mathfrak S}_{0}({\mathbf 1}_{]0,1]}, \alpha) {\mathfrak S}_1(\alpha) \frac{\Psi(x, y)^3}{x} 
 + o_{\ee, \Phi}\left(\frac{\Psi(x, y)^3}{x}\right) \]
ce qui implique bien la majoration
\[ N(x, y) \leq {\mathfrak S}_{0}({\mathbf 1}_{]0,1]}, \alpha) {\mathfrak S}_1(\alpha) \frac{\Psi(x, y)^3}{x}  + o_{\ee, \Phi}\left(\frac{\Psi(x, y)^3}{x}\right) \]
Cela montre la première assertion du Théorème~\ref{th2}, compte tenu de
\[ {\mathfrak S}_{0}({\mathbf 1}_{]0,1]}, \alpha) {\mathfrak S}_1(\alpha) \gg 1 .\]

\bibliographystyle{smfplain}
\bibliography{smooth_abc}
 
\end{document}